\documentclass[leqno,11pt,a4paper]{amsart}

\usepackage[margin=2.5cm]{geometry}
\usepackage{graphicx}
\usepackage[usenames,dvipsnames]{color}
\usepackage{hyperref}
\usepackage{mathabx}
\usepackage{mleftright}
\usepackage{booktabs}
\usepackage{colortbl}
\usepackage{multirow}
\usepackage{pdflscape}
\usepackage{tikz}
\usepackage{tikz-qtree}
\usepackage{longtable}
\usetikzlibrary{decorations.markings}
\newtheorem{thm}{Theorem}[section]
\newtheorem{lemma}[thm]{Lemma}
\newtheorem{cor}[thm]{Corollary}
\newtheorem{prop}[thm]{Proposition}
\theoremstyle{definition}
\newtheorem{example}[thm]{Example}
\newtheorem{remark}[thm]{Remark}

\newtheorem{definition}[thm]{Definition}

\numberwithin{equation}{section}
\newtheorem{case}{Case}
\newtheorem{subcase}{Case}
\numberwithin{subcase}{case}

\numberwithin{subsubcase}{subcase}

\numberwithin{subsubsubcase}{subsubcase}
\newtheorem*{examplecase}{Example Case}
\newcommand{\orig}{\mathbf{0}}
\newcommand{\Z}{\mathbb{Z}}
\newcommand{\Q}{\mathbb{Q}}

\newcommand{\Proj}{\mathbb{P}}
\newcommand{\Hom}[1]{\mathrm{Hom}\mleft({#1}\mright)}

\newcommand{\abs}[1]{\left\vert{#1}\right\vert}
\newcommand{\mult}[1]{\mathrm{mult}\mleft({#1}\mright)}
\newcommand{\intr}[1]{\mathrm{int}\mleft({#1}\mright)}
\newcommand{\V}[1]{\mathrm{vert}\mleft({#1}\mright)}

\newcommand{\dual}[1]{{#1}^\vee}
\newcommand{\bdual}[1]{\dual{\mleft({#1}\mright)}}
\newcommand{\conv}[1]{\mathrm{conv}\mleft({#1}\mright)}
\newcommand{\sconv}[1]{\mathrm{conv}\mleft\{{#1}\mright\}}
\newcommand{\scone}[1]{\mathrm{cone}\mleft\{{#1}\mright\}}

\newcommand{\hmin}{{h_\mathrm{min}}}
\newcommand{\hmax}{{h_\mathrm{max}}}

\newcommand{\mut}{\mathrm{mut}}
\newcommand{\Hilb}[1]{\mathrm{Hilb}\mleft({#1}\mright)}

\newcommand{\NQ}{N_\Q}
\newcommand{\MQ}{M_\Q}
\renewcommand{\gcd}[1]{\mathrm{gcd}\mleft\{{#1}\mright\}}
\renewcommand{\min}[1]{\mathrm{min}\mleft\{{#1}\mright\}}

\newcommand{\Xn}[3]{X_{#1}^{(#2,#3)}}
\newcommand{\GL}{\mathrm{GL}}
\graphicspath{{images/}}
\begin{document}
\author[T.~Coates]{Tom Coates}
\author[S.~Gonshaw]{Samuel Gonshaw}
\author[A.~M.~Kasprzyk]{Alexander Kasprzyk}
\author[N.~Nabijou]{Navid Nabijou}
\address{Department of Mathematics\\Imperial College London\\London, SW$7$\ $2$AZ\\UK}
\email{t.coates@imperial.ac.uk}
\email{samuel.gonshaw10@imperial.ac.uk}
\email{a.m.kasprzyk@imperial.ac.uk}
\email{{navid.nabijou09@imperial.ac.uk}}
\title{Mutations of Fake Weighted Projective Spaces}
\begin{abstract} 
We characterise  mutations between fake weighted projective spaces, and give explicit formulas for how the weights and multiplicity change under mutation. In particular, we prove that multiplicity-preserving  mutations between fake weighted projective spaces are mutations over edges of the corresponding simplices. As an application, we analyse the canonical and terminal fake weighted projective spaces of maximal degree.
\end{abstract}
\maketitle
\section{Introduction}\label{sec:introduction}
In this paper we analyse mutations between fake weighted projective spaces; equivalently we analyse mutations between lattice simplices. Mutations arise naturally when considering mirror symmetry for Fano manifolds. A Fano manifold $X$ is expected to correspond under mirror symmetry to a Laurent polynomial~\cite{EHX,BCKvS98,HV,BCKvS00,Ba04,Auroux,ProcECM,QC105}. In general there will be many different Laurent polynomials which correspond to a given Fano manifold, and it is expected that these Laurent polynomials are related via birational transformations analogous to cluster transformations~\cite{FZ,GU10,ACGK12,GHK13}. These cluster-style transformations act on Newton polytopes via mutations; see Definition~\ref{defn:mutation} below. A mutation between Newton polytopes can be thought of as the tropicalisation of the corresponding cluster-type transformation. Ilten~\cite{Ilt12} has shown that if two lattice polytopes $P$ and $Q$ are related by mutation then the corresponding toric varieties $X_P$ and $X_Q$ are deformation equivalent: there exists a flat family $\mathcal{X}\rightarrow\Proj^1$ such that $\mathcal{X}_0\cong X_P$ and $\mathcal{X}_\infty\cong X_Q$. Mutations are thus expected to form the one-skeleton of the tropicalisation of the Hilbert scheme. Our understanding of this one-skeleton is rudimentary but improving~\cite{ACGK12,AK13}; in this paper we conduct the first systematic analysis in higher dimensions.

The notion of mutation raises many new and interesting combinatorial questions: for example, how can polytopes be classified up to mutation, and what properties of polytopes are mutation-invariant? Here we begin to address these questions by analysing the behaviour of lattice simplicies under mutation. In two dimensions, Akhtar and Kasprzyk determined how the weights of a fake weighted projective plane, i.e.~the weights of a lattice triangle, change under mutation, and showed that mutations between fake weighted projective planes are multiplicity-preserving~\cite{AK13}. We show below that the situation in higher dimensions is different. We give an explicit formula (Theorem~\ref{mutationformula}) for how the weights of a fake weighted projective space, i.e.~the weights of a lattice simplex, change under mutation, and derive a strong necessary condition (Theorem~\ref{onlyedges}) for a mutation to preserve multiplicity.
In~\S\S\ref{sec:canonical_sings}--\ref{sec:terminal_sings} we apply our results to the study of fake weighted projective spaces of high degree with canonical and terminal singularities.

\subsection*{Acknowledgements}

The authors thank Mohammad Akhtar and Alessio Corti for useful
conversations. This research is supported by the Royal Society, ERC
Starting Investigator Grant number~240123, the Leverhulme Trust, and
EPSRC grant EP/I008128/1.

\section{Fake Weighted Projective Space} \label{sec:fake_wps}
We begin by recalling some standard definitions. Throughout let
$N\cong\Z^n$ denote a lattice of rank $n$ with dual lattice
$M:=\Hom{N,\Z}$. From the toric viewpoint $N$ corresponds to the
lattice of one-parameter subgroups and $M$ corresponds to the lattice
of characters. For an introduction to toric geometry see~\cite{Dan78}.

\begin{definition} \label{defn:Fano_polytope}
  A convex lattice polytope $P\subset\NQ:=N\otimes_\Z\Q$ is said to be \emph{Fano} if:
  \begin{enumerate}
  \item\label{item:nondegenerate}
    $P$ is of maximum dimension, that is $\dim{P}=n$;
  \item\label{item:Fano_orig}
    the origin is contained in the strict interior of $P$, that is $\orig\in\intr{P}$;
  \item\label{item:Fano_primitive}
    the vertices $\V{P}$ of $P$ are primitive lattice points.
  \end{enumerate}
\end{definition}

In addition to being compelling combinatorial objects, Fano polytopes
are in bijective correspondence with toric Fano varieties:
see~\cite{KN12} for an overview. The complete fan in $N$ generated by
the faces of $P$, which we call the \emph{spanning fan} of $P$,
corresponds to a toric Fano variety $X_P$, that is, to a (possibly
singular) projective toric variety with ample anticanonical divisor
$-K_X$. Two Fano polytopes $P$ and $P'$ give isomorphic varieties
$X_P\cong X_{P'}$ if and only if there exists a change of basis of the
underlying lattice $N$ sending $P$ to $P'$. Thus we regard $P$ as
being defined only up to $\GL_n(\Z)$-equivalence.

\begin{definition} \label{defn:weights}
  Let $P:=\sconv{v_0,\ldots,v_n}\subset\NQ$ be a Fano $n$-simplex. By
  Definition~\ref{defn:Fano_polytope}\eqref{item:Fano_orig} there
  exists a unique choice of $n+1$ coprime positive integers
  $\lambda_0,\ldots,\lambda_n\in\Z_{>0}$ such that
  $\lambda_0v_0+\ldots+\lambda_nv_n=\orig$. These are called the
  \emph{(reduced) weights} of $P$. 
\end{definition}

Definition~\ref{defn:Fano_polytope}\eqref{item:Fano_primitive} implies
that any $n$ of the weights are also coprime, or in other words that
the weights are \emph{well-formed}. See~\cite[\S5]{I-F00} for details
of the natural role that reduced and well-formed weights play in the
study of weighted projective space.

\begin{definition}\label{defn:fake_wps} 
  Let $P\subset\NQ$ be a Fano
  $n$-simplex and let $N':=v_0\cdot\Z+\ldots+v_n\cdot\Z$ be the
  sublattice in $N$ generated by the vertices of $P$. The rank-one
  $\Q$-factorial toric Fano variety $X_P$ given by the spanning fan of
  $P$ is $X_P=\Proj(\lambda_0,\ldots,\lambda_n) / (N/N')$, where the
  group $N/N'$ acts freely in codimension one. We call $X_P$ a
  \emph{fake weighted projective space}. 
\end{definition}

Fake weighted projective spaces have been studied
in~\cite{Con02,Buc08,Kas09}. One important invariant is the
\emph{multiplicity}: the index of the sublattice $N'$ in $N$, denoted
by $\mult{P}:=[N:N']$. A fake weighted projective space is a weighted
projective space if and only if
$\mult{P}=1$~\cite[Proposition~2]{BB92}.

\section{Mutations} \label{sec:mutations}
We recall the definition of mutation, following~\cite[\S3]{ACGK12}. A
primitive element $w\in M$ determines a surjective linear map $w
\colon N\rightarrow\Z$ which extends naturally to a map
$\NQ\rightarrow\Q$. A point $v\in\NQ$ is said to be at \emph{height}
$w(v)$. Given a subset $S\subset\NQ$, if $w(v)=h$ for all $v\in S$ we
say that $S$ lies at height $h$ and write $w(S)=h$. The hyperplane
$H_{w,h}$ is defined to be the set of all points in $\NQ$ at height
$h$. For a convex lattice polytope $P\subset\NQ$ we define
$w_h(P):=\conv{H_{w,h}\cap P\cap N}$ to be the (possibly empty) convex
hull of all lattice points in $P$ at height $h$. We set
$\hmin:=\min{w(v)\mid v\in P}$ to be the minimum height occurring
amongst the the points of $P$, and $\hmax$ to be the maximum height.
Since $P$ is a lattice polytope, both $\hmin$ and $\hmax$ are
integers. If $\orig\in\intr{P}$, and in particular if $P$ is Fano,
then $\hmin<0$ and $\hmax>0$.

\begin{definition}\label{defn:factor}
  A \emph{factor} of $P\subset\NQ$ with respect to a primitive height
  function $w\in M$ is a lattice polytope $F\subset\NQ$ such that:
  \begin{enumerate} 
  \item $w(F)=0$; 
  \item for every integer $h$ with $\hmin\leq h<0$, there exists a (possibly empty) lattice polytope
    $G_h\subset\NQ$ such that
    $H_{w,h}\cap\V{P}\subseteq G_h+(-h)F\subseteq w_h(P)$.
  \end{enumerate} 
\end{definition}

\begin{definition}\label{defn:mutation}
  The \emph{(combinatorial) mutation} of $P\subset\NQ$ with respect to
  the primitive height function $w\in M$ and a factor $F\subset\NQ$ is
  the convex lattice polytope
  \[
  \mut_w(P,F):=\conv{\bigcup_{h=\hmin}^{-1}G_h\cup\bigcup_{h=0}^\hmax\left(w_h(P)+hF\right)}\subset\NQ.
  \]
\end{definition}

Although this is not obvious from the definition, the mutation
$\mut_w(P,F)$ is independent of the choice of the
$G_h$~\cite[Proposition~1]{ACGK12}. Furthermore, $\mut_w(P,F)$ is a
Fano polytope if and only if $P$ is a Fano
polytope~\cite[Proposition~2]{ACGK12}. More obviously, mutations are
always invertible (if $Q:=\mut_w(P,F)$ then
$P\cong\mut_{-w}(Q,F)$)~\cite[Lemma~2]{ACGK12}, and translating the
factor results in an isomorphic mutation (i.e.
$\mut_w(P,F)\cong\mut_w(P,F+v)$ for any $v\in N$ with $w(v)=0$).

Mutations have a natural description as transformations of the dual polytope
\[
\dual{P}:=\{u\in\MQ\mid u(v)\geq -1\text{ for all $v\in P$}\}.
\]
A mutation induces a piecewise $\GL_n(\Z)$-transformation $\varphi
\colon u\mapsto u-u_\mathrm{min} w$ of $\MQ$ such that
$\bdual{\varphi(\dual{P})}=\mut_w(P,F)$; here
$u_\mathrm{min}:=\min{u(v)\mid v\in F}$. This is analogous to a cluster
transformation. As a consequence $\abs{k\dual{P}\cap
  M}=\abs{k\dual{Q}\cap M}$ for any dilation $k\in\Z_{\geq 0}$, where
$Q:=\mut_w(P,F)$~\cite[Proposition~4]{ACGK12}; hence
$\Hilb{X_P,-K_{X_P}}=\Hilb{X_Q,-K_{X_Q}}$ and $X_P$ and $X_Q$ have the
same anticanonical degree.

\begin{example}
  \label{exa:P1113P1146}
  The weighted projective spaces $\Proj(1,1,1,3)$ and $\Proj(1,1,4,6)$
  have the largest degree amongst all canonical\footnote{See~\S\ref{sec:canonical_and_terminal_sings} below for a discussion of canonical singularities.} \emph{toric} Fano
  threefolds~\cite{Kas08a} and amongst all \emph{Gorenstein} canonical
  Fano threefolds~\cite{Pro05}. They are related by a mutation
  \cite[Example 7]{ACGK12}. The simplex associated to
  $\Proj(1,1,1,3)$ is
  $P:=\sconv{(1,0,0),(0,1,0),(0,0,1),(-1,-1,-3)}\subset\NQ$. Setting
  $w = (-1,2,0)\in M$ gives $\hmin = {-1}$ and $\hmax = 2$, with
  $w_{-1}(P)$ equal to the edge $\sconv{(-1,-1,-3),(1,0,0)}$ and
  $w_{2}(P)$ given by the vertex $(0,1,0)$. The factor
  \[
  F:=\sconv{(0,0,0),(2,1,3)}
  \]
  gives:
  \[
  \mut_w(P,F) = \sconv{(-1,-1,-3),(0,0,1),(0,1,0),(4,3,6)}
  \]
  and this is the simplex associated with $\Proj(1,1,4,6)$. This mutation is illustrated in Figure~\ref{fig:maximal_degree}.
\end{example}

\begin{figure}[htbp]
\centering
\includegraphics{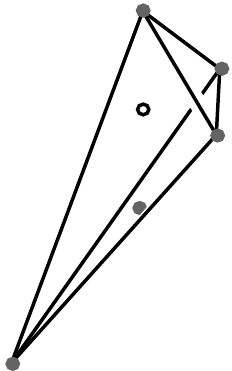}
\raisebox{73pt}{\hspace{15pt}$\longmapsto$\hspace{-10pt}}
\includegraphics{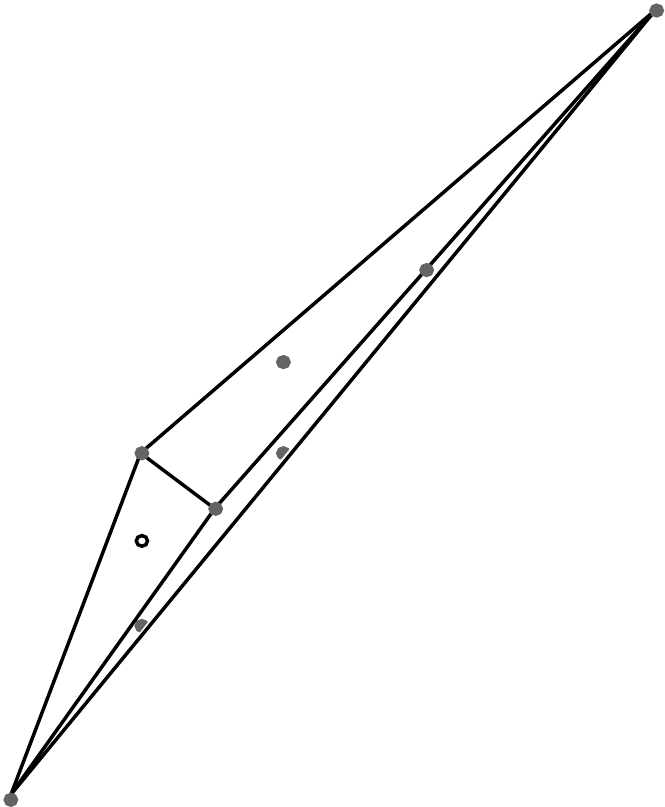}
\caption{An edge mutation from the polytope corresponding to $\Proj(1,1,1,3)$, depicted on the left, to the polytope corresponding to $\Proj(1,1,4,6)$. This mutation is described in Example~\ref{exa:P1113P1146}.}
\label{fig:maximal_degree}
\end{figure}

\section{Mutations of $n$-Simplices}\label{sec:basic_formulae}

We begin by establishing some basic properties of  mutations
between simplices. Throughout this section we assume that $P$ is a
Fano $n$-simplex, and that $w \in M$ and $F \subset \NQ$ are
(respectively) a primitive height function and a factor such that
$Q := \mut_w(P,F)$ is a simplex. In other words, we assume that there
is a  mutation from the fake weighted projective space $X$
associated with $P$ to the fake weighted projective space $Y$
associated with $Q$.

\begin{lemma}\label{ncasesprop}
  Let $P$ be a Fano $n$-simplex, let $w \in M$ be a primitive height
  function, and let $F \subset \NQ$ be a factor such that $Q :=
  \mut_w(P,F)$ is a simplex. Suppose that the mutation
  from $P$ to $Q$ is non-trivial, so that $P \not \cong Q$. Then
  $w_\hmax(P)$ is a vertex of $P$ and $F$ is a translation of
  $\frac{1}{\abs{\hmin}}w_\hmin(P)$.
\end{lemma} 

\begin{proof}
  We first consider $w_\hmax(P)$. Suppose for a contradiction that
  $w_\hmax(P)$ is not a vertex of $P$. Then in particular it contains
  an edge $E_1$ of $P$. Let us pick an edge $E_2$ contained in
  $w_\hmin(P)$ (if such an edge does not exist then $w_\hmin(P)$ is a
  vertex and the mutation is trivial). Note that $E_1$ and $E_2$
  cannot be parallel: if they were then the four endpoints of
  $E_1$ and $E_2$ would lie in a common two-dimensional affine
  subspace and thus would be affinely dependent; this contradicts the fact that
  $P$ is a simplex. It is an immediate consequence of the definition
  of mutation that $w_\hmax(Q) = w_\hmax(P) +\hmax F$. Recall that $F$
  is a Minkowski factor of $w_\hmin(P)$. Because $P$ is a simplex,
  $w_\hmin(P)$ is also a simplex and so by~\cite[Result 13]{Shep63}
  $F$ is a dilation and translation of $w_\hmin(P)$. It follows that
  $\hmax F$ is also a dilation and translation of $w_\hmin(P)$. Let
  $E_2^\prime$ denote the edge of $\hmax F$ corresponding to the edge
  $E_2$ of $w_\hmin(P)$. Then $E_1$ and $E_2^\prime$ are not parallel,
  because $E_1$ and $E_2$ are not parallel. Thus the face $E_1 +
  E_2^\prime$ of $w_\hmax(Q) = w_\hmax(P) +\hmax F$ is a
  quadrilateral. Since $Q$ is by assumption a simplex, and therefore
  all faces of $Q$ are simplices, this gives a contradiction. We
  conclude that $w_\hmax(P)$ consists of a single vertex.

  Now consider the factor $F$. By the definition of factor there exists a lattice polytope $G_\hmin$ such that
  \begin{equation*} 
    w_\hmin(P) = G_\hmin +\abs{\hmin} F.
  \end{equation*}
  We claim that $G_\hmin$ is a point. After mutation we have
  $w_\hmin(Q) = G_\hmin$. As discussed above, we can mutate $Q$ back
  to $P$ by taking the height function $-w$ and the factor $F$. But
  $(-w)_{-\hmin}(Q) = w_\hmin(Q) = G_\hmin$ and we conclude from the
  first part of the proposition that $G_\hmin$ is a point: $G_\hmin =
  \{v\}$. It follows that $\abs{\hmin}F = w_\hmin(P) - v$, and so $F$
  is a translation of $\frac{1}{\abs{\hmin}}w_\hmin(P)$ as
  required.
\end{proof}

\begin{definition} 
  If $w_\hmin(P)$ has $k+1$ vertices then we say that
  the corresponding mutation is a \emph{mutation over a $k$-face}.
  Cases of particular interest are $k=1$ and $k=n-1$, which we call
  \emph{mutations over edges} and \emph{mutations over facets},
  respectively. 
\end{definition}

\begin{lemma}\label{zeroheight}
  Let $P$ be a Fano $n$-simplex, let $w \in M$ be a primitive height
  function, and let $F \subset \NQ$ be a factor such that $Q :=
  \mut_w(P,F)$ is a simplex. Let $v\in\V{P}$ be such that $w(v)
  \ne \hmax$ and $w(v) \ne \hmin$. Then $w(v)=0$. 
\end{lemma}

\begin{proof}
  Let us write $P =\sconv{ v_0,\ldots, v_n}$ and $w_\hmin(P)
  =\sconv{v_1 ,\ldots, v_k}$. If $k=n$ then
  the statement holds vacuously, so let us assume that $k<n$. Without
  loss of generality we may assume that $\orig\in\V{F}$. In view
  of Lemma~\ref{ncasesprop} we may assume further that
  $w_\hmax(P) = \{v_0\}$, that $w_\hmin(P)
  = v_1+\abs{\hmin}F$, and that
  \begin{equation*} 
    F :=\sconv{\orig,\frac{1}{\abs{\hmin}}(v_2 - v_1),\ldots,\frac{1}{\abs{\hmin}}(v_k - v_1)}.
  \end{equation*}
  Suppose for a contradiction that there exists some $v\in\V{P}$
  such that $w(v) \ne \hmax$, $w(v) \ne \hmin$, and $w(v)\neq 0$.
  Without loss of generality we can take $v=v_n$.

  Suppose first that $w(v_n)>0$. Let $w^\prime\in M$ be a
  primitive lattice point such that $w^\prime(v_i)=0$ for $v_i\in\{
  v_1,\ldots, v_{n-1}\}$. We can choose $w^\prime$ so that
  $w^\prime(v_0)>0$ and $w^\prime(v_n)<0$. Let
  $h_\text{max}^\prime=\mathrm{sup}\{w^\prime(p)\mid p\in P\}$ and
  $h_\text{min}^\prime=\mathrm{inf}\{w^\prime(p)\mid p\in P\}$. We see
  that $H_{w^\prime,h_\text{max}^\prime}\cap P =\{v_0\}$ and
  $H_{w^\prime,h_\text{min}^\prime}\cap P =\{v_n\}$. Note that by the
  definition of $F$, $w^\prime(F)=0$ and so $w^\prime(x) = w^\prime(x
  + w(x)F)$ for all $x\in P$ (with $w(x)\geq 0$). Then
  $H_{w^\prime,h_\text{max}^\prime}\cap Q = v_0 +\hmax F = v_0 +
  w(v_0)F$ and $H_{w^\prime,h_\text{min}^\prime}\cap Q = v_n +
  w(v_n)F$. Thus $v_0+w(v_0)F$ and $v_n+w(v_n)F$ are two faces of $Q$,
  with vertices $v_0+w(v_0)f$ and $v_n+w(v_n)f$ for $f\in\V{F}$.
  Let $f_1,f_2\in\V{F}$ be any two distinct vertices of $F$; these
  exist since otherwise $F$ is a point and the mutation is trivial. As
  $Q$ is a simplex 
  \begin{equation*}
    F^\prime:=\sconv{v_0+w(v_0)f_1,v_0+w(v_0)f_2,v_n+w(v_n)f_1,v_n+w(v_n)f_2}
  \end{equation*}
  is a two-dimensional face of $Q$. But $F^\prime$ has four vertices and
  thus is not a simplex. This gives a contradiction.

  Suppose instead that $w(v_n)<0$. By~\cite[Lemma~3.7]{ACGK12} there
  exists a $v_Q\in\V{Q}$ and $v_F\in\V{F}$ such that
  $v_n=v_Q-w(v_Q)v_F$. Applying $w$ to this equation yields
  $w(v_Q)=w(v_n)<0$. We now have $\mut_{-w}(Q,F) = P$ but $-w(v_Q)>0$
  which, by the preceding argument, gives a contradiction. It follows
  that $w(v_n) = 0$.
\end{proof}

Combining the previous two results gives a necessary and sufficient
combinatorial condition for the existence of a mutation between
simplices.

\begin{lemma}\label{iffmut}
Let $P$ be a Fano $n$-simplex and let $w \in M$ be a primitive height
  function. Let the vertices of $P$ be $\{v_0,v_1,\ldots,v_n\}$,
  ordered such that $w_\hmin(P) =\sconv{v_1,\ldots, v_k}$. There
  exists a factor $F \subset \NQ$ such that $Q :=
  \mut_w(P,F)$ is a simplex if and only if the following hold:
  \begin{enumerate}
  \item\label{item:iffmut_1} $w_{\hmax}(P)$ is a vertex;
  \item\label{item:iffmut_2} $\hmin\divides v_i - v_1$ for $i\in\{1,\ldots, k\}$;
  \item\label{item:iffmut_3} if $v$ is a vertex of $P$ such that $w(v) \ne \hmax$ and
    $w(v) \ne \hmin$, then $w(v) = 0$. 
  \end{enumerate} 
\end{lemma}

\begin{proof} 
  The ``only if'' direction is  Lemmas~\ref{ncasesprop} 
  and~\ref{zeroheight}. On the other 
  hand, it is clear that if conditions~\eqref{item:iffmut_1}--\eqref{item:iffmut_3} are satisfied, then
  defining 
  \begin{equation*}
    F :=\sconv{\orig,\frac{1}{\abs{\hmin}}(v_2 -
      v_1),\ldots,\frac{1}{\abs{\hmin}}(v_k - v_1)}
  \end{equation*}
  gives a factor with respect to $w$. Let us label the vertices of
  $P=\sconv{v_0,\ldots,v_n}$ so that: 
  \begin{align*}
    w(v_0)=\hmax, &&
    w(v_1)=\ldots=w(v_k) =\hmin, &&
    w(v_{k+1})=\ldots=w(v_n)=0.
    \end{align*} 
    From~\cite[Lemma~3.7]{ACGK12} we have that 
  $\V{Q}\subseteq\{v_0^\prime,v_1^\prime,\ldots,v_n^\prime\}$, where:
  \[
  v_i^\prime=
  \begin{cases}
    v_i&\text{ if }i=0, i=1,\text{ or }i\in\{k+1,\ldots,n\}\\
    v_0+\frac{\hmax}{\abs{\hmin}}(v_i-v_1)&\text{ if }i\in\{2,\ldots,k\}.
  \end{cases}
  \]
Furthermore by~\cite[Proposition~3.11]{ACGK12} we have that $Q$ is
Fano, and hence is of maximal dimension in the $n$-dimensional
lattice. It follows that $Q$ must have at least $n+1$ vertices. Thus
$Q=\sconv{v_0^\prime,v_1^\prime,\ldots,v_n^\prime}$, and so $Q$ is a
simplex.
\end{proof}

\begin{example}
  \label{exa:P1148P1114}
  Consider the $3$-simplex $P =\sconv{ (1,-1,0), (-2,-2,-1),
    (-2,-2,1), (0,1,0) }$. The weighted projective space
  associated to $P$ is $\Proj(1,1,4,8)$. Let $w = (1, 0, 0)\in
  M$. Lemma~\ref{iffmut} implies that there is a factor $F$ such
  that $Q:= \mut_w(P,F)$ is the simplex:
    \begin{equation*} Q
      =\sconv{(1,-1,0),(-2,-2,-1),(1,-1,1),(0,1,0)}
    \end{equation*} 
    The weighted projective space associated to $Q$ is
    $\Proj(1,1,1,4)$.
\end{example}

The main result of this paper is:

\begin{thm}\label{mutationformula}
  Let $X$ and $Y$ be fake weighted projective spaces related by a
   mutation. Suppose that $P=\sconv{v_0,\ldots,v_n}\subset\NQ$ is the simplex corresponding to
  $X$, that $Q \subset \NQ$ is the simplex corresponding to $Y$, and that the weights of $X$ are
  $\lambda_0,\lambda_1,\ldots,\lambda_n$.
  Let $w\in M$ be the primitive height function and $F \subset \NQ$ be
  the factor such that $Q = \mut_w(P,F)$.
  Then we may relabel the vertices of $P$ such that $w_\hmax(P)=v_0$,
  $w_\hmin(P)=\sconv{v_1,\ldots,v_k}$, and the weights of $Y$ are:
  \begin{equation*} \frac{1}{d}\left(\lambda_0\lambda_1,(\lambda_1+\ldots+\lambda_k)^2,\lambda_0\lambda_2,\ldots,\lambda_0\lambda_k
      ,\lambda_{k+1}(\lambda_1+\ldots+\lambda_k),\ldots,\lambda_n(\lambda_1+\ldots+\lambda_k)\right)
\end{equation*}
  where $d$ is a positive integer satisfying: 
  \begin{equation*}
    d\cdot\frac{\mult{X}}{\mult{Y}}=\frac{\lambda_0^{k-1}}{(\lambda_1
      +\ldots +\lambda_k)^{k-2}} \end{equation*} 
\end{thm}

\begin{proof}
  By Lemma~\ref{iffmut} we have $Q=\sconv{v_0^\prime,v_1^\prime,\ldots,v_n^\prime}$ where:
  \begin{equation}
    \label{eq:vprime}
    v_i^\prime=
    \begin{cases}
      v_i&\text{ if }i=0, i=1,\text{ or }i\in\{k+1,\ldots,n\}\\
      v_0+\frac{\hmax}{\abs{\hmin}}(v_i-v_1)&\text{ if }i\in\{2,\ldots,k\}.
    \end{cases}
  \end{equation}
  Since $P$ is Fano we may assume that the weights
  $\lambda_0,\lambda_1,\ldots,\lambda_n$ are well-formed. We normalise
  by setting $h:=\sum_{i=0}^n\lambda_i$ and $\lambda_i^\prime
  :=\frac{1}{h}\lambda_i$. Then $\sum_{i=0}^n\lambda_i^\prime v_i
  =\orig$, $\sum_{i=0}^n\lambda_i^\prime = 1$, and
  $\lambda_i^\prime\geq 0$ for all $i$. The sequence
  $\lambda^\prime_0,\lambda^\prime_1,\ldots,\lambda^\prime_n$ is
  unique with these properties; these are the \emph{normalised barycentric
    co-ordinates} for $P$.

  Since $Q$ is Fano there exist $\mu_0^\prime,\mu_1^\prime,\ldots,\mu_n^\prime$ such
  that $\sum_{i=0}^n\mu_i^\prime v_i^\prime =\orig$,
  $\sum_{i=0}^n\mu_i^\prime = 1$, and $\mu_i^\prime\geq 0$ for all
  $i$. From $\sum_{i=0}^n\mu_i^\prime v_i^\prime =\orig$ and
  \eqref{eq:vprime} we have that:
  \[
  \left(\mu_0^\prime +\sum_{i=2}^k\mu_i^\prime\right) v_0
  +\left(\mu_1^\prime
    -\frac{\hmax}{\abs{\hmin}}\sum_{i=2}^k\mu_i^\prime\right) v_1
  +\sum_{i=2}^k\mu_i^\prime\frac{\hmax}{\abs{\hmin}} v_i
  +\sum_{i=k+1}^n\mu_i^\prime v_i =\orig.
  \]
Let $\theta_i$ denote the coefficient of $v_i$ in the expression
above. We claim that the $\theta_i$ are normalised barycentric
co-ordinates. It is clear that $\sum_{i=0}^n\theta_i
=\sum_{i=0}^n\mu_i^\prime = 1$, and that
$\theta_0,\theta_2,\theta_3,\ldots,\theta_n\geq 0$. It remains to
check that $\theta_1\geq 0$. Suppose for a contradiction that
$\theta_1 < 0$. Then we have:
\begin{equation*} 
-\theta_1 v_1 =\theta_0 v_0 +\theta_2 v_2 +\ldots
  +\theta_n v_n\in\scone{v_0,v_2,\ldots ,v_n}
\end{equation*} 
and $-\theta_1 > 0$, so a point on the ray from $\orig$ through $v_1$
lies in the cone over $v_0, v_2,\ldots,v_n$. This contradicts the
fact that $\orig$ lies in the strict interior of $P$.

Hence the $\theta_i$ are normalised barycentric co-ordinates for $P$,
and so by uniqueness we have $\theta_i =\lambda_i^\prime$ for all $i$.
Solving these equations for  $\mu_i^\prime$ yields:
\[
\mu_i^\prime=
\begin{cases}
\lambda_0^\prime -\frac{\abs{\hmin}}{\hmax}\sum_{i=2}^k\lambda_i^\prime&\text{ if }i=0\\
\lambda_1^\prime +\sum_{i=2}^k\lambda_i^\prime&\text{ if }i=1\\
\frac{\abs{\hmin}}{\hmax}\lambda_i^\prime&\text{ if }i\in\{ 2,\ldots, k\}\\
\lambda_i^\prime&\text{ if }i\in\{ k+1,\ldots, n\}.
\end{cases}
\]
Applying $w$ to both sides of the equation
$\sum_{i=1}^n\lambda_iv_i=\orig$, we find that
\begin{equation} 
  \label{hweqn}
  \frac{\abs{\hmin}}{\hmax} =\frac{\lambda_0}{\lambda_1 +\ldots +\lambda_k}
\end{equation}
and thus:
\[
\mu_i^\prime=
\begin{cases}
\lambda_0^\prime -\frac{\lambda_0}{\lambda_1 +\ldots +\lambda_k}\sum_{i=2}^k\lambda_i^\prime&\text{ if }i=0\\
\lambda_1^\prime +\sum_{i=2}^k\lambda_i^\prime&\text{ if }i=1\\
\frac{\lambda_0}{\lambda_1 +\ldots +\lambda_k}\lambda_i^\prime&\text{ if }i\in\{ 2,\ldots, k\}\\
\lambda_i^\prime&\text{ if }i\in\{ k+1,\ldots, n\}.
\end{cases}
\]
Now we can form integer weights by defining $\mu_i = h(\lambda_1
+\ldots +\lambda_k)\mu_i^\prime$; this gives:
\[
\mu_i=
\begin{cases}
\lambda_0\lambda_1&\text{ if }i=0\\
(\lambda_1 +\ldots +\lambda_k)^2&\text{ if }i=1\\
\lambda_0\lambda_i&\text{ if }i\in\{2,\ldots,k\}\\
\lambda_i (\lambda_1 +\ldots +\lambda_k)&\text{ if }i\in\{ k+1 ,\ldots , n\}.
\end{cases}
\]
These weights are integers, but they may not be well-formed. However since $Q$ is Fano, we know that the weights are well-formed if and only if they are reduced. Therefore it remains to divide through by their greatest common divisor, which we denote by $d$. Thus the weights of $Q$ are:
\begin{equation*}
\frac{1}{d}\left(\lambda_0\lambda_1,(\lambda_1+\ldots+\lambda_k)^2,\lambda_0\lambda_2,\ldots,\lambda_0\lambda_k,\lambda_{k+1}(\lambda_1+\ldots+\lambda_k),\ldots,\lambda_n(\lambda_1+\ldots+\lambda_k)\right).
\end{equation*}

Consider now the degrees of $X$ and $Y$: 
\begin{align*}
  &
  (-K_X)^n=
  \frac{(\lambda_0+\ldots+\lambda_n)^n}{\lambda_0\ldots\lambda_n \, \mult{X}} \\
  & (-K_Y)^n=
\frac{\frac{1}{d^n}\left(\lambda_0\lambda_1+\ldots+\lambda_n(\lambda_1+\ldots+\lambda_k)\right)^n}{\frac{1}{d^{n+1}}\left((\lambda_0\lambda_1)\ldots\lambda_n(\lambda_1+\ldots+\lambda_k) \right)\,\mult{Y}}
\end{align*}
Since degree is preserved by mutation, we conclude that
\begin{equation}\label{dmult}
d\cdot\frac{\mult{X}}{\mult{Y}}=\frac{\lambda_0^{k-1}}{(\lambda_1+\ldots+\lambda_k)^{k-2}}
\end{equation}
as claimed.
\end{proof}

\begin{remark} \label{rem:multiplicity_preserving}
  When $\mult{X}=\mult{Y}$, and in particular if $X$ and $Y$ are weighted projective spaces, we have:
  \begin{equation*}
    d=\frac{\lambda_0^{k-1}}{(\lambda_1 +\ldots +\lambda_k)^{k-2}}
  \end{equation*}
  This gives an explicit expression for the weights after mutation in
  terms of the weights before mutation.
\end{remark}

\begin{remark}
  In the case of a mutation over a facet we see that the new weights
  are: 
  \begin{equation*} \frac{1}{d}\left(\lambda_0\lambda_1,(\lambda_1+\ldots+\lambda_k)^2,\lambda_0\lambda_2,\ldots,\lambda_0\lambda_n\right)
  \end{equation*}
  Note that $\lambda_0$ divides $d$, because
  $\lambda_0\divides\lambda_0\lambda_i$ for $i = 1,\ldots, n$ and the
  weights are well-formed. On the other hand, after dividing
  through by $\lambda_0$ we obtain well-formed weights, and so in fact
  $d=\lambda_0$. In this case, therefore, we obtain an explicit formula for how the multiplicity changes:
  \begin{equation*}
    \frac{\mult{X}}{\mult{Y}}=\left(\frac{\lambda_0}{\lambda_1+\ldots+\lambda_n}\right)^{n-2}
  \end{equation*}
\end{remark}

\section{Multiplicity-preserving Mutations}\label{sec:mult_preserving}

The following result places a strong restriction on which mutations of
fake weighted projective spaces can preserve multiplicity.

\begin{thm}\label{onlyedges}
  Any non-trivial multiplicity-preserving mutation between fake
  weighted projective spaces $X$ and $Y$ is a mutation over an edge.
\end{thm}

\begin{proof}
  Let $P = \sconv{v_0,v_1,\ldots,v_n}$ be the simplex associated to
  $X$, and let $\lambda_0,\lambda_1,\ldots,\lambda_n$ be the
  corresponding weights. Let $Q =
  \sconv{v_0^\prime,v_1^\prime,\ldots,v_n^\prime}$ be the simplex
  associated to $Y$, and let
  $\lambda^\prime_0,\lambda^\prime_1,\ldots,\lambda^\prime_n$ be the
  corresponding weights. Suppose for a contradiction that $P$ and $Q$
  are related by a non-trivial mutation over a $k$-face, for some
  $k>2$. By Remark~\ref{rem:multiplicity_preserving}, after
  reordering weights if necessary, we have:
  \begin{align*} (\mu_0,\ldots & ,\mu_n) =\\
    &\frac{1}{d}\left((\lambda_1+\ldots+\lambda_k)^2,\lambda_0\lambda_1,\lambda_0\lambda_2,\ldots,\lambda_0\lambda_k,\lambda_{k+1}(\lambda_1+\ldots+\lambda_k),\ldots,\lambda_n(\lambda_1+\ldots+\lambda_k)\right)\end{align*}
  where:
  \begin{equation*}
    d=\frac{\lambda_0^{k-1}}{(\lambda_1 +\ldots +\lambda_k)^{k-2}}
  \end{equation*}
  We recall from \eqref{hweqn} that
  \begin{equation*}\frac{\hmax}{\abs{\hmin}}=\frac{\lambda_1+\ldots+\lambda_k}{\lambda_0}\end{equation*}
  and write $\hmax/\abs{\hmin} = A/B$ with $A$ and $B$ coprime
  integers. So, for $i\in\{ 1,\ldots, k\}$, we have:
  \begin{equation*}
    \mu_i=\frac{\lambda_0\lambda_i}{d}=\lambda_i\left(\frac{\lambda_1+\ldots+\lambda_k}{\lambda_0}\right)^{k-2}=\lambda_i\left(\frac{A}{B}\right)^{k-2}
  \end{equation*}
  Since $A$ and $B$ are coprime and $k > 2$, we have that
  $B\divides\lambda_i$. Similarly for $i\in\{k+1,\ldots,n\}$ we have:
  \begin{equation*}
    \mu_i=\frac{\lambda_i(\lambda_1+\ldots+\lambda_k)}{d}=\lambda_i\left(\frac{\lambda_1+\ldots+\lambda_k}{\lambda_0}\right)^{k-1}=\lambda_i\left(\frac{A}{B}\right)^{k-1}
  \end{equation*}
  and so in this case too $B\divides\lambda_i$. However
  $\lambda_1,\lambda_2,\ldots,\lambda_n$ are coprime, because the
  weights of $X$ are well-formed, and therefore $B=1$.

  Since $\hmax/\abs{\hmin} = A/B = A$ is an integer, we have that
  $\lambda_0\divides\lambda_1 +\ldots +\lambda_k$. It follows that
  $\lambda_0\divides d\mu_i$ for all $i$, and since the weights
  $\mu_i$ are reduced we conclude that $\lambda_0\divides d$. Then
  \begin{equation*}
    \frac{d}{\lambda_0}=\left(\frac{\lambda_0}{\lambda_1 +\ldots +\lambda_k}\right)^{k-2}
  \end{equation*}
  is an integer. Taking the $(k-2)$th root (recall that $k>2$) we see
  that $\frac{\lambda_0}{\lambda_1 +\ldots +\lambda_k}$ is an integer,
  and hence that $\lambda_1 +\ldots +\lambda_k\divides\lambda_0$. Thus
  $\lambda_1 +\ldots +\lambda_k=\lambda_0$. Substituting this into
  our expression for the $\mu_i$ shows that the mutation is trivial,
  which is a contradiction.
\end{proof}

\begin{cor}
  \label{cor:wps_weight_divisibility}
  Suppose that $X$ is a weighted projective space that admits a
  non-trivial mutation to another weighted projective space. Let
  $\lambda_0$ be the weight corresponding to the vertex $w_\hmax(P)$,
  and let $\lambda_1$,~$\lambda_2$ be the weights corresponding to the
  vertices of the edge $w_\hmin(P)$. Then:
  \begin{enumerate}
  \item\label{item:wps_weight_divisibility_1} $\lambda_0\divides(\lambda_1+\lambda_2)^2$
  \item\label{item:wps_weight_divisibility_2} $\gcd{\lambda_1,\lambda_2}\divides\lambda_0$
  \end{enumerate}
\end{cor}
\begin{proof}
  The mutation is an edge mutation, and so
  Theorem~\ref{mutationformula} implies both that $d=\lambda_0$
  and that $d \divides (\lambda_1+\lambda_2)^2$. This proves~\eqref{item:wps_weight_divisibility_1}.
  Looking again at Theorem~\ref{mutationformula} and using
  well-formedness of weights, we see that $\gcd{\lambda_1,\lambda_2}
  \divides d$. This proves~\eqref{item:wps_weight_divisibility_2}.
\end{proof}

\begin{remark}
  Let $X$ be a fake weighted projective plane. Akhtar and Kasprzyk
  characterise mutations from $X$ to other fake weighted projective
  planes in terms of solutions to an associated Diophantine
  equation~\cite[Proposition~3.12]{AK13}. Their argument relies on
  the fact that, for lattice triangles, the square-free parts of the
  weights are preserved (up to reordering) under mutation. This
  phenomenon does not persist in higher dimensions:
  \begin{enumerate}
  \item Example~\ref{exa:P1113P1146} above shows that, in general,
    neither the square-free parts of the weights nor the
    $k$th-power-free parts of the weights nor the square parts of the
    weights nor the $k$th-power parts of the weights are preserved.
  \item Example~\ref{exa:P1148P1114} above shows that, in general,
    neither the $n$th-power-free parts nor the $n$th-power parts of
    the weights are preserved.
  \end{enumerate}
\end{remark}

\section{Canonical and Terminal Singularities}\label{sec:canonical_and_terminal_sings}

Terminal and canonical singularities were introduced by Miles Reid; they play a fundamental role in birational geometry~\cite{Rei80,Rei83,Rei87}. Terminal singularities form the smallest class of singularities that must be allowed if one wishes to construct minimal models in dimensions three or more. Canonical singularities can be regarded as the limit of terminal singularities; they arise naturally as the singularities occurring on canonical models of varieties of general type. From the toric viewpoint, terminal and canonical singularities have a particularly elegant combinatorial description. A toric singularity corresponds to a strictly convex rational polyhedral cone $\sigma\subset\NQ$~\cite{Dan78}. The cone $\sigma$ is \emph{terminal} if and only if:
\begin{enumerate}
\item\label{item:Q_Gorenstein}
the lattice points $\rho_1,\ldots,\rho_m$ corresponding to the primitive generators of the rays of $\sigma$ are contained in an affine hyperplane $H_u:=\{v\in\NQ\mid u(v)=1\}$ for some $u\in M_\Q$;
\item\label{item:terminal_hyper}
with the exception of the origin $\orig$ and the generators $\rho_i$ of the rays, no other lattice points of $N$ are contained in the part of $\sigma$ on or under $H_u$, i.e.
$$N\cap\sigma\cap\{v\in\NQ\mid u(v)\leq 1\}=\{\orig,\rho_1,\ldots,\rho_m\}$$
\end{enumerate}
The cone $\sigma$ is \emph{canonical} if and only if~\eqref{item:Q_Gorenstein} holds and
\begin{enumerate}
\item[(ii$'$)]\label{item:canonical_hyper}
the origin $\orig$ is the only lattice point contained in the part of $\sigma$ under $H_u$, i.e.
$$N\cap\sigma\cap\{v\in\NQ\mid u(v)<1\}=\{\orig\}$$
\end{enumerate}
When the hyperplane $H_u$ in condition~\eqref{item:Q_Gorenstein} corresponds to a lattice point $u\in M$, the singularity is called \emph{Gorenstein}.

\section{Weighted Projective Spaces With Canonical Singularities}\label{sec:canonical_sings}

We now apply our results to the study of weighted projective spaces with canonical singularities, giving some geometric context, in terms of mutations, for a recent combinatorial result that characterises the fake weighted projective spaces of maximal degree with at worst canonical singularities~\cite{AKN13}. We say that a (fake) weighted projective space with at worst canonical singularities is a \emph{canonical} (fake) weighted projective space. 

\begin{remark}
  If $\Proj(\lambda_0,\ldots,\lambda_n)/G$ is a canonical fake weighted projective space then $\Proj(\lambda_0,\ldots,\lambda_n)$ has canonical singularities too. Thus a canonical fake weighted projective space of maximal degree is necessarily a weighted projective space.
\end{remark}

\begin{definition} \label{def:Sylvester}
  The \emph{Sylvester numbers} $y_0,y_1,y_2,\ldots$ are defined by:
  \[
  y_n = 
  \begin{cases}
    2&\text{ if }n=0\\
    1 +\prod_{i=0}^{n-1} y_i&\text{ otherwise.}
  \end{cases}
  \]
  We set $t_n := y_n - 1$.
\end{definition}

\begin{lemma}[\cite{Syl80}] \ 
  \label{lem:sylv} 
  \begin{enumerate}
  \item\label{item:sylv_1} The Sylvester numbers are pairwise coprime;
  \item\label{item:sylv_2} If $n\geq 1$ and $i< n$ then $t_n/y_i$ is an integer. Furthermore if $p$ is a prime dividing $y_i$ then $p\divides t_n/y_j$ for $j< n$,~$j\neq i$, and $p\notdivides t_n/y_i$;
  \item\label{item:sylv_3} We have:
    \[
    \sum_{i=0}^{n-1}\frac{1}{y_i} =\frac{t_n-1}{t_n}.
    \]
  \end{enumerate}
\end{lemma}

Define:
\[
X_n := \textstyle\Proj\left(1,1,\frac{2 t_{n-1}}{y_{n-2}},\ldots, \frac{2 t_{n-1}}{y_0}\right)
\]
$X_n$ is an $n$-dimensional canonical weighted projective space. Averkov, Kr\"umpelmann, and Nill~\cite{AKN13} have proved that, if $n \geq 4$, then $X_n$ is the unique $n$-dimensional canonical weighted projective space of maximum degree; this generalizes the corresponding result for Gorenstein weighted projective spaces, which is due to Nill~\cite{Nill07}. In dimension three there are precisely two canonical weighted projective spaces of maximum degree, and these are connected by a mutation: see Example~\ref{exa:P1113P1146}. We next determine all weighted projective spaces $\Xn{n}{m}{a}$ that might be connected to $X_n$ by a sequence of mutations through weighted projective spaces, and show that none of the $\Xn{n}{m}{a}$ are canonical.

\begin{prop}\label{canonicalwpsmutationgraph} 
  Define integers $\lambda_i^{(m,a)}$, where $0 \leq i \leq n$ and $0
  \leq a \leq n-2$, by:
  \begin{align*}
    \lambda_i^{(0,a)} & = 
    \begin{cases}
      1&\text{ if }i=0\\
      \frac{2t_{n-1}}{y_a}&\text{ if }i=1\\
      1&\text{ if }i=2\\
      \frac{2t_{n-1}}{y_{k_i}}&\text{ otherwise}
    \end{cases}
    \intertext{where $\{k_3,\ldots,k_n\}$ is $\{0,1,\ldots,\widehat{a},\ldots,n-2\}$ with $a$
      omitted, and:} 
    \lambda_i^{(m,a)} &=
    \begin{cases}
      \lambda_2^{(m-1,a)}&\text{ if }i=0\\
      \lambda_1^{(m-1,a)}&\text{ if }i=1\\
      \frac{\big(\lambda_1^{(m-1,a)}+\lambda_2^{(m-1,a)}\big)^2}{\lambda_0^{(m-1,a)}}&\text{ if }i=2\\
      \frac{\lambda_i^{(m-1,a)}\big(\lambda_1^{(m-1,a)}+\lambda_2^{(m-1,a)}\big)}{\lambda_0^{(m-1,a)}} 
      &\text{ otherwise}
    \end{cases}
  \end{align*}
  for $m \geq 1$. Set:
  \[
  \Xn{n}{m}{a} :=
  \Proj\left(\lambda_0^{(m,a)},\lambda_1^{(m,a)},\ldots,\lambda_n^{(m,a)}\right)
  \]
  Then:
  \begin{enumerate}
  \item\label{item:canonicalwpsmutationgraph_1} $\Xn{n}{0}{a} = X_n$ for all $a$;
  \item\label{item:canonicalwpsmutationgraph_2} if $m>0$ and if $X$ is a weighted projective space with a
    non-trivial mutation to $\Xn{n}{m}{a}$ then either $X =
    \Xn{n}{m-1}{a}$ or $X = \Xn{n}{m+1}{a}$;
  \item\label{item:canonicalwpsmutationgraph_3} $\Xn{n}{m}{a}$ is not canonical for any $m \geq 1$.
  \end{enumerate}
  In other words, the graph of mutations between weighted projective
  spaces starting at $X_n$ is a subtree of the following graph:

\begin{figure}[htbp]
  \centering
  \begin{tikzpicture}[grow'=down,sibling distance=50pt,label/.style={%
      postaction={ decorate,transform shape,
        decoration={ markings, mark=at position .5 with\node #1;}}}]]
    \Tree 
    [.$X_n$ 
	[.$\Xn{n}{1}{n-2}$ 
		[.$\Xn{n}{2}{n-2}$
			 [.$\Xn{n}{3}{n-2}$ {\vdots}
			]
		]
	]
        [.$\cdots$\vphantom{$\Xn{n}{1}{n-2}$}
		[.$\cdots$\vphantom{$\Xn{n}{1}{n-2}$}
			 [.$\cdots$\vphantom{$\Xn{n}{1}{n-2}$} {\vdots}
			]
		]
	]
	[.$\Xn{n}{1}{1}$ 
		[.$\Xn{n}{2}{1}$
			 [.$\Xn{n}{3}{1}$ {\vdots}
			]
		]
	]
	[.$\Xn{n}{1}{0}$ 
		[.$\Xn{n}{2}{0}$
			 [.$\Xn{n}{3}{0}$ {\vdots}
			]
		]
	]
]
\end{tikzpicture}  
  \label{fig:canonical_graph}
\end{figure}

\noindent and the only canonical weighted projective space in this
graph is $X_n$.
\end{prop}

\begin{proof} Statement~\eqref{item:canonicalwpsmutationgraph_1} is trivial. For~\eqref{item:canonicalwpsmutationgraph_3}, recall that a
  weighted projective space $\Proj(\mu_0,\ldots,\mu_n)$ is canonical
  if and only if
  \[
  \sum_{i=0}^{n}\left\{\frac{\mu_i\kappa}{h}\right\}\in\{1,\ldots,n-1\}
  \]
  for every $\kappa\in\{2,\ldots,h-2\}$, where $\{x\}$ denotes the
  fractional part of $x$ and $h$ is the sum of the
  $\mu_i$~\cite[Proposition~2.5]{Kas13}. This fails for $m \geq 1$ and
  $\kappa = h^{(m,a)}-\lambda_1^{(m,a)}-\lambda_2^{(m-1,a)}$, where $h^{(m,a)}$ is
  the sum of the weights of $\Xn{n}{m}{a}$: see Lemma~\ref{lem:not_canonical}.

  It remains to
  prove~\eqref{item:canonicalwpsmutationgraph_2}. Suppose first that
  there is a non-trivial mutation from $X_n$ to a weighted projective
  space $X$. By Theorem~\ref{onlyedges}, this must be an edge
  mutation. Let $P$ be the simplex associated to $X_n$, let $\mu_0$
  denote the weight associated to the vertex $w_\hmax(P)$, and let
  $\mu_1$,~$\mu_2$ be the weights associated to the vertices of
  $w_\hmin(P)$. We consider the possible values of
  $\mu_0$,~$\mu_1$,~$\mu_2$.

  \begin{case} $\mu_0\neq 1$. Then $\mu_0 = 2t_{n-1}/y_{a_0}$ for
    some $a_0\in\{ 0,\ldots, n-2\}$.
    \begin{subcase} $\mu_1 = 1$,~$\mu_2 = 1$. Since $n\geq 4$, there
      exist $a_1$,~$a_2 \in \{0,1,\ldots,n-2\}$ distinct such that
      $y_{a_1} \ne y_{a_0}$ and $y_{a_2} \ne y_{a_0}$. At least one of
      $y_{a_1}$ or $y_{a_2}$ is not equal to $2$, and hence is a divisor
      of $\mu_0$ not equal to 2 or 4. But $(\mu_1+\mu_2)^2=4$, and this
      contradicts Corollary~\ref{cor:wps_weight_divisibility}\eqref{item:wps_weight_divisibility_1}.
    \end{subcase}

    \begin{subcase} $\mu_1\neq 1$,~$\mu_2 = 1$. Then $\mu_1 =
      2t_{n-1}/y_{a_1}$ for some $a_1\neq a_0$. Choose $a_2$ not equal
      to $a_0$ or $a_1$, and let $p$ be a prime dividing
      $y_{a_2}$. Lemma~\ref{lem:sylv}\eqref{item:sylv_2} imples that $p$ divides both
      $2t_{n-1}/y_{a_0}$ and $2t_{n-1}/y_{a_1}$. So $p$ does not divide
      $(2t_{n-1}/y_{a_1} +1)^2$, and this contradicts
      Corollary~\ref{cor:wps_weight_divisibility}\eqref{item:wps_weight_divisibility_1}.
    \end{subcase}

    \begin{subcase} $\mu_1 \ne 1$,~$\mu_2\neq 1$. Then $\mu_1 =
      2t_{n-1}/y_{a_1}$ and $\mu_2 = 2t_{n-1}/y_{a_2}$ for
      $a_1$,~$a_2$ distinct and not equal to $a_0$. Suppose first that
      $y_{a_0}\neq 2$, and let $p$ be a prime dividing
      $y_{a_0}$. Lemma~\ref{lem:sylv}\eqref{item:sylv_2} implies that
      $p\divides\gcd{\mu_1,\mu_2}$ and $p\notdivides\mu_0$,
      contradicting Corollary~\ref{cor:wps_weight_divisibility}\eqref{item:wps_weight_divisibility_2}. On
      the other hand if $y_{a_0} = 2$ then the same argument with
      $p=2$ shows that $4\divides\gcd{\mu_1,\mu_2}$ and
      $4\notdivides\mu_0$, again contradicting
      Corollary~\ref{cor:wps_weight_divisibility}\eqref{item:wps_weight_divisibility_2}
    \end{subcase}
  \end{case}

  \begin{case} $\mu_0 = 1$.
    \begin{subcase} $\mu_1 \ne 1$, $\mu_2 = 1$. Then
      $\mu_1=2t_{n-1}/y_a$ for some $a\in\{ 0,\ldots, n-2\}$, and
      Theorem~\ref{mutationformula} implies that $X =
      \Xn{n}{1}{a}$.
    \end{subcase}

    \begin{subcase}$\mu_1 \ne 1$, $\mu_2 \ne 1$. Then
      $\mu_1=2t_{n-1}/y_{a_1}$,$\mu_2=2t_{n-1}/y_{a_2}$ for
      $a_1$,~$a_2$ distinct. Choose $a_0$ not equal to $a_1$ or
      $a_2$, and let $p$ be a prime dividing
      $y_{a_0}$. Lemma~\ref{lem:sylv}\eqref{item:sylv_2} implies that
      $p\divides\gcd{\mu_1,\mu_2}$ and $p\notdivides\mu_0$,
      contradicting Corollary~\ref{cor:wps_weight_divisibility}\eqref{item:wps_weight_divisibility_2}.
    \end{subcase}
  \end{case}
  \noindent This completes the proof in the case where $m=0$.

  Suppose now that $m>1$, and that there is a non-trivial mutation
  from $\Xn{n}{m}{a}$ to a weighted projective space $X$.
  Theorem~\ref{onlyedges} implies that this is an edge
  mutation. Let $P$ be the simplex corresponding to $\Xn{n}{m}{a}$,
  let $\mu_0$ denote the weight associated to the vertex $w_\hmax(P)$,
  and let $\mu_1$,~$\mu_2$ denote the weights associated to the
  vertices of $w_\hmin(P)$. To declutter the notation, write
  $\lambda_i$ for the weight $\lambda_i^{(m,a)}$ of $\Xn{n}{m}{a}$.
  We consider the possible values of $\mu_0$,~$\mu_1$,~$\mu_2$ in
  turn.
  
  \begin{examplecase}
    $\mu_0=\lambda_i$ with $i \ge
    3$,~$\mu_1=\lambda_1$,~$\mu_2=\lambda_2$. By Lemma~\ref{claim3}
    we have that $\mu_0$ and $\mu_2$ share a common factor that does
    not divide $\mu_1$, and so does not divide $\mu_1+\mu_2$. Thus
    $\mu_0\notdivides (\mu_1+\mu_2)$, contradicting
    Corollary~\ref{cor:wps_weight_divisibility}\eqref{item:wps_weight_divisibility_1}.
  \end{examplecase}
  
  The remaining cases are entirely analogous. We arrive at a
  contradiction in all but two cases, as summarized in
  Table~\ref{tab:summary_mgt1} below; here
  $\lambda_i$,~$\lambda_j$,~$\lambda_k$ denote distinct elements of
  $\{3,\ldots,n\}$. The case
  $\mu_0=\lambda_0$,~$\mu_1=\lambda_1$,~$\mu_2=\lambda_2$ yields
  $\Xn{n}{m+1}{a}$, and the case
  $\mu_0=\lambda_2$,~$\mu_1=\lambda_0$,~$\mu_2=\lambda_1$ yields
  $\Xn{n}{m-1}{a}$. This completes the proof in the case where $m>1$.

  \begin{table}[ht]
    \centering
    \small
    \begin{tabular}{cccllccccll} \toprule
      $\mu_0$ & $\mu_1$ & $\mu_2$ & \multicolumn{1}{c}{Contradicts} &
      \multicolumn{1}{c}{Using} & \qquad&
      $\mu_0$ & $\mu_1$ & $\mu_2$ & \multicolumn{1}{c}{Contradicts} &
      \multicolumn{1}{c}{Using} \\
      \cmidrule(r{0.3ex}){1-5} 
      \cmidrule(l{0.3ex}){7-11} 
      $\lambda_0$ & $\lambda_1$ & $\lambda_2$ & \multicolumn{1}{c}{--} & \multicolumn{1}{c}{--}&&
      $\lambda_2$ & $\lambda_1$ & $\lambda_i$ &Corollary~\ref{cor:wps_weight_divisibility}\eqref{item:wps_weight_divisibility_1} & Lemma~\ref{claim2}\\
      $\lambda_0$ & $\lambda_1$ & $\lambda_i$ &Corollary~\ref{cor:wps_weight_divisibility}\eqref{item:wps_weight_divisibility_1} & Lemma~\ref{claim9}&&
      $\lambda_2$ & $\lambda_i$ & $\lambda_j$ &Corollary~\ref{cor:wps_weight_divisibility}\eqref{item:wps_weight_divisibility_2} & Lemma~\ref{claim11}\\
      $\lambda_0$ & $\lambda_2$ & $\lambda_i$ &Corollary~\ref{cor:wps_weight_divisibility}\eqref{item:wps_weight_divisibility_1} & Lemma~\ref{claim9}&&
      $\lambda_i$ & $\lambda_0$ & $\lambda_1$ &Corollary~\ref{cor:wps_weight_divisibility}\eqref{item:wps_weight_divisibility_1} & Lemma~\ref{claim2}\\
      $\lambda_0$ & $\lambda_i$ & $\lambda_j$ &Corollary~\ref{cor:wps_weight_divisibility}\eqref{item:wps_weight_divisibility_2} & Lemma~\ref{claim11}&&
      $\lambda_i$ & $\lambda_0$ & $\lambda_2$ &Corollary~\ref{cor:wps_weight_divisibility}\eqref{item:wps_weight_divisibility_1} & Lemma~\ref{claim2}\\
      $\lambda_1$ & $\lambda_0$ & $\lambda_2$ &Corollary~\ref{cor:wps_weight_divisibility}\eqref{item:wps_weight_divisibility_1} & Lemma~\ref{claim10}&&
      $\lambda_i$ & $\lambda_0$ & $\lambda_j$ &Corollary~\ref{cor:wps_weight_divisibility}\eqref{item:wps_weight_divisibility_1} & Lemma~\ref{claim4}\\
      $\lambda_1$ & $\lambda_0$ & $\lambda_i$ &Corollary~\ref{cor:wps_weight_divisibility}\eqref{item:wps_weight_divisibility_1} & Lemma~\ref{claim8}&&
      $\lambda_i$ & $\lambda_1$ & $\lambda_2$ &Corollary~\ref{cor:wps_weight_divisibility}\eqref{item:wps_weight_divisibility_1} & Lemma~\ref{claim2}\\
      $\lambda_1$ & $\lambda_2$ & $\lambda_i$ &Corollary~\ref{cor:wps_weight_divisibility}\eqref{item:wps_weight_divisibility_1} & Lemma~\ref{claim8}&&
      $\lambda_i$ & $\lambda_1$ & $\lambda_j$ &Corollary~\ref{cor:wps_weight_divisibility}\eqref{item:wps_weight_divisibility_1} & Lemma~\ref{claim4}\\
      $\lambda_1$ & $\lambda_i$ & $\lambda_j$ &Corollary~\ref{cor:wps_weight_divisibility}\eqref{item:wps_weight_divisibility_1} & Lemma~\ref{claim8}&&
      $\lambda_i$ & $\lambda_2$ & $\lambda_j$ &Corollary~\ref{cor:wps_weight_divisibility}\eqref{item:wps_weight_divisibility_1} & Lemma~\ref{claim7}\\
      $\lambda_2$ & $\lambda_0$ & $\lambda_1$ & \multicolumn{1}{c}{--} & \multicolumn{1}{c}{--}&&
      $\lambda_i$ & $\lambda_j$ & $\lambda_k$ &Corollary~\ref{cor:wps_weight_divisibility}\eqref{item:wps_weight_divisibility_1} & Lemma~\ref{claim8}\\
      $\lambda_2$ & $\lambda_0$ & $\lambda_i$ &Corollary~\ref{cor:wps_weight_divisibility}\eqref{item:wps_weight_divisibility_1} & Lemma~\ref{claim2}\\
      \bottomrule
    \end{tabular}
    \vspace{0.5em}
    \caption{A summary of the argument in the case where $m>1$.}
    \label{tab:summary_mgt1}
  \end{table}
  
  Suppose now that $m=1$. We can argue exactly as for $m>1$, except
  in those cases where Lemma~\ref{claim7} is used. We consider these
  three cases separately. As before, write $\lambda_i$ for
  $\lambda_i^{(m,a)}$.

  \begin{case}
    \label{canonical_wps_case_3}
    $\mu_0 = \lambda_0$,~$\mu_1 = \lambda_1$,~$\mu_2 = \lambda_i$ with
    $i \geq 3$. Then $\mu_0 = 1$, $\mu_1 = 2t_{n-1}/y_a$, and $\mu_2
    = \frac{2t_{n-1}}{y_{a_i}}\left(\frac{2t_{n-1}}{y_a}+1\right)$.
    Choose $a_j$ not equal to $a$ or $a_i$, and let $p$ be a prime
    dividing $y_{a_j}$. Then $p$ divides $\mu_1$ and $\mu_2$ but does
    not divide $\mu_0$, contradicting
    Corollary~\ref{cor:wps_weight_divisibility}\eqref{item:wps_weight_divisibility_2}.
  \end{case}

  \begin{case}
    \label{canonical_wps_case_4}
    $\mu_0 =\lambda_0$,~$\mu_1 =\lambda_2$,~$\mu_2 =\lambda_i$ with $i
    \geq 3$. Then $\mu_0=1$, $\mu_1 = (2t_{n-1}/y_a+1)^2$, and $\mu_2
    =\frac{2t_{n-1}}{y_{a_i}}\left(\frac{2t_{n-1}}{y_a}+1\right)$.
    Let $p$ be a prime dividing $2t_{n-1}/y_a + 1$. Then $p$ divides
    $\mu_1$ and $\mu_2$ but does not divide $\mu_0$, contradicting
    Corollary~\ref{cor:wps_weight_divisibility}\eqref{item:wps_weight_divisibility_2}.
  \end{case}
  
  \begin{case}
    \label{canonical_wps_case_5}
    $\mu_0 =\lambda_i$ with $i \geq 3$,~$\mu_1=\lambda_2$,~$\mu_2
    =\lambda_j$ with $j \geq 3$. Then $\mu_0
    =\frac{2t_{n-1}}{y_{a_i}}\left(\frac{2t_{n-1}}{y_a}+1\right)$,
    $\mu_1 = (2t_{n-1}/y_a+1)^2$, and $\mu_2
    =\frac{2t_{n-1}}{y_{a_j}}\left(\frac{2t_{n-1}}{y_a}+1\right)$.
    Let $p$ be a prime dividing $y_a$. Then $p$ divides $\mu_0$ and
    $\mu_2$ (Lemma~\ref{lem:sylv}) but does not divide $\mu_1$
    (Lemma~\ref{sylv++}). This contradicts Corollary~\ref{cor:wps_weight_divisibility}\eqref{item:wps_weight_divisibility_1}.
  \end{case}
  \noindent This completes the proof in the case where $m=1$.
\end{proof}

\section{Weighted Projective Spaces With Terminal Singularities}\label{sec:terminal_sings}

A (fake) weighted projective space with at worst terminal singularities is called a \emph{terminal} (fake) weighted projective space. As before, a terminal fake weighted projective space of maximal degree is necessarily a weighted projective space. We now give a conjectural classification of terminal weighted projective spaces of maximal degree, and analyse them in terms of mutations.

In dimension~$3$, the unique terminal weighted projective space of
maximum degree is $\Proj^3$. In dimension~$4$, Kasprzyk has
shown~\cite[Lemma~3.5]{Kas13} that the unique terminal weighted
projective space of maximum degree is:
\begin{equation*}
  \Proj(1,1,6,14,21)=\textstyle\Proj\left(1,1,\frac{t_3}{y_2},\frac{t_3}{y_1},\frac{t_3}{y_0}\right)
\end{equation*}
The classification problem for terminal weighted projective spaces of
maximum degree in dimensions~$5$ and higher is open. The space
\[
X_n=\textstyle\Proj\left(1,1,\frac{t_{n-1}}{y_{n-2}},\ldots,\frac{t_{n-1}}{y_0}\right)
\]
is terminal~\cite[Lemma~3.7]{Kas13} and it seems reasonable to
conjecture that, for $n \geq 4$, $X_n$ is the unique terminal weighted
projective space of maximum degree. Note that the methods
of~\cite{AKN13} do not apply to this problem, as they are unable to
distinguish between canonical and terminal singularities (they are
insensitive to lattice points on the boundary of a simplex).

\begin{prop}\label{terminalwpsmutationgraph} 
  Define integers $\lambda_i^{(m,a)}$, where $0 \leq i \leq n$ and $0
  \leq a \leq n-2$, by:
  \begin{align*}
    \lambda_i^{(0,a)} & = 
    \begin{cases}
      1&\text{ if }i=0\\
      \frac{t_{n-1}}{y_a}&\text{ if }i=1\\
      1&\text{ if }i=2\\
      \frac{t_{n-1}}{y_{k_i}}&\text{ otherwise}
    \end{cases}
    \intertext{where $\{k_3,\ldots,k_n\}$ is $\{0,1,\ldots,\widehat{a},\ldots,n-2\}$ with $a$
      omitted, and:} 
    \lambda_i^{(m,a)} &=
    \begin{cases}
      \lambda_2^{(m-1,a)}&\text{ if }i=0\\
      \lambda_1^{(m-1,a)}&\text{ if }i=1\\
      \frac{\big(\lambda_1^{(m-1,a)}+\lambda_2^{(m-1,a)}\big)^2}{\lambda_0^{(m-1,a)}}&\text{ if }i=2\\
      \frac{\lambda_i^{(m-1,a)}\big(\lambda_1^{(m-1,a)}+\lambda_2^{(m-1,a)}\big)}{\lambda_0^{(m-1,a)}} 
      &\text{ otherwise}
    \end{cases}
  \end{align*}
  for $m \geq 1$. Suppose that $n \geq 5$. Set:
  \[
  \Xn{n}{m}{a} :=
  \Proj\left(\lambda_0^{(m,a)},\lambda_1^{(m,a)},\ldots,\lambda_n^{(m,a)}\right)
  \]
  Then:
  \begin{enumerate}
  \item\label{item:terminalwpsmutationgraph_1} $\Xn{n}{0}{a} = X_n$ for all $a$;
  \item\label{item:terminalwpsmutationgraph_2} if $m>0$ and if $X$ is a weighted projective space with a
    non-trivial mutation to $\Xn{n}{m}{a}$ then either $X =
    \Xn{n}{m-1}{a}$ or $X = \Xn{n}{m+1}{a}$;
  \item\label{item:terminalwpsmutationgraph_3} $\Xn{n}{m}{a}$ is not terminal for any $m \geq 1$.
  \end{enumerate}
  In other words, the graph of mutations between weighted projective
  spaces starting at $X_n$ is a subtree of the following graph:

\begin{figure}[htbp]
  \centering
  \begin{tikzpicture}[grow'=down,sibling distance=50pt,label/.style={%
      postaction={ decorate,transform shape,
        decoration={ markings, mark=at position .5 with\node #1;}}}]]
    \Tree 
    [.$X_n$ 
	[.$\Xn{n}{1}{n-2}$ 
		[.$\Xn{n}{2}{n-2}$
			 [.$\Xn{n}{3}{n-2}$ {\vdots}
			]
		]
	]
        [.$\cdots$\vphantom{$\Xn{n}{1}{n-2}$}
		[.$\cdots$\vphantom{$\Xn{n}{1}{n-2}$}
			 [.$\cdots$\vphantom{$\Xn{n}{1}{n-2}$} {\vdots}
			]
		]
	]
	[.$\Xn{n}{1}{1}$ 
		[.$\Xn{n}{2}{1}$
			 [.$\Xn{n}{3}{1}$ {\vdots}
			]
		]
	]
	[.$\Xn{n}{1}{0}$ 
		[.$\Xn{n}{2}{0}$
			 [.$\Xn{n}{3}{0}$ {\vdots}
			]
		]
	]
]
\end{tikzpicture}  
  \label{fig:terminal_graph}
\end{figure}

\noindent and the only terminal weighted projective space in this
graph is $X_n$.
\end{prop}

\begin{proof}
  Statement~\eqref{item:terminalwpsmutationgraph_1} is trivial.
  For~\eqref{item:terminalwpsmutationgraph_3}, recall that a weighted
  projective space $\Proj(\mu_0,\ldots,\mu_n)$ is terminal if and only
  if
  \begin{equation}\label{terminaliff}
    \sum_{i=0}^{n}\left\{\frac{\mu_i\kappa}{h}\right\}\in\{2,\ldots,n-1\}
  \end{equation}
  for every $\kappa\in\{2,\ldots,h-2\}$, where $\{x\}$ denotes the
  fractional part of $x$ and $h$ is the sum of the
  $\mu_i$~\cite[Proposition~2.3]{Kas13}. This fails for $m \geq 1$
  and $\kappa = h^{(m,a)}-h^{(m-1,a)}$, where $h^{(m,a)}$ is the sum
  of the weights of $\Xn{n}{m}{a}$: see Lemma~\ref{lem:not_terminal}.

  It remains to prove~\eqref{item:terminalwpsmutationgraph_2}. This is
  analogous to the proof of
  Proposition~\ref{canonicalwpsmutationgraph}\eqref{item:terminalwpsmutationgraph_2}.
  The analogs of Lemmas~\ref{claim1}--\ref{claim7} and
  Lemmas~\ref{claim8},~\ref{claim10} hold for these weights
  $\lambda_i^{(m,a)}$ too, with almost identical proofs, and
  Lemma~\ref{claim11term} functions as a replacement for
  Lemma~\ref{claim11}. Suppose first that $m > 1$ and that there is a
  non-trivial mutation from $\Xn{n}{m}{a}$ to a weighted projective
  space $X$. By Theorem~\ref{onlyedges}, this must be an edge
  mutation. Let $P$ be the simplex associated to $\Xn{n}{m}{a}$, let
  $\mu_0$ denote the weight associated to the vertex $w_\hmax(P)$, and
  let $\mu_1$,~$\mu_2$ be the weights associated to the vertices of
  $w_\hmin(P)$. To declutter the notation, we again write $\lambda_i$
  for the weight $\lambda_i^{(m,a)}$ of $\Xn{n}{m}{a}$. We consider
  the possible values of $\mu_0$,~$\mu_1$,~$\mu_2$ in turn, arriving
  at a contradiction in all but two cases. This is summarized in
  Table~\ref{tab:summary_terminal_mgt1} below, where
  $\lambda_i$,~$\lambda_j$,~$\lambda_k$ denote distinct elements of
  $\{3,\ldots,n\}$. The case
  $\mu_0=\lambda_0$,~$\mu_1=\lambda_1$,~$\mu_2=\lambda_2$ yields
  $\Xn{n}{m+1}{a}$, and the case
  $\mu_0=\lambda_2$,~$\mu_1=\lambda_0$,~$\mu_2=\lambda_1$ yields
  $\Xn{n}{m-1}{a}$.

\begin{table}[ht]
    \centering
    \small
    \begin{tabular}{cccllccccll} \toprule
      $\mu_0$ & $\mu_1$ & $\mu_2$ & \multicolumn{1}{c}{Contradicts} &
      \multicolumn{1}{c}{Using} & \qquad&
      $\mu_0$ & $\mu_1$ & $\mu_2$ & \multicolumn{1}{c}{Contradicts} &
      \multicolumn{1}{c}{Using} \\
      \cmidrule(r{0.3ex}){1-5} 
      \cmidrule(l{0.3ex}){7-11} 
     $\lambda_0$ & $\lambda_1$ & $\lambda_2$ & \multicolumn{1}{c}{--}
      & \multicolumn{1}{c}{--} &&
      $\lambda_2$ & $\lambda_1$ & $\lambda_i$
      &Corollary~\ref{cor:wps_weight_divisibility}\eqref{item:wps_weight_divisibility_1}
      & Lemma~\ref{claim2}\\
      $\lambda_0$ & $\lambda_1$ & $\lambda_i$
      &Corollary~\ref{cor:wps_weight_divisibility}\eqref{item:wps_weight_divisibility_1}
      & Lemma~\ref{claim9}&&
      $\lambda_2$ & $\lambda_i$ & $\lambda_j$
      &Corollary~\ref{cor:wps_weight_divisibility}\eqref{item:wps_weight_divisibility_2}
      & Lemma~\ref{claim11term}\\
      $\lambda_0$ & $\lambda_2$ & $\lambda_i$
      &Corollary~\ref{cor:wps_weight_divisibility}\eqref{item:wps_weight_divisibility_1}
      & Lemma~\ref{claim9}&&
      $\lambda_i$ & $\lambda_0$ & $\lambda_1$
      &Corollary~\ref{cor:wps_weight_divisibility}\eqref{item:wps_weight_divisibility_1}
      & Lemma~\ref{claim2}\\
      $\lambda_0$ & $\lambda_i$ & $\lambda_j$
      &Corollary~\ref{cor:wps_weight_divisibility}\eqref{item:wps_weight_divisibility_2}
      & Lemma~\ref{claim11term}&&
     $\lambda_i$ & $\lambda_0$ & $\lambda_2$
      &Corollary~\ref{cor:wps_weight_divisibility}\eqref{item:wps_weight_divisibility_1}
      & Lemma~\ref{claim2}\\
      $\lambda_1$ & $\lambda_0$ & $\lambda_2$
      &Corollary~\ref{cor:wps_weight_divisibility}\eqref{item:wps_weight_divisibility_1}
      & Lemma~\ref{claim10}&&
     $\lambda_i$ & $\lambda_0$ & $\lambda_j$
      &Corollary~\ref{cor:wps_weight_divisibility}\eqref{item:wps_weight_divisibility_1}
      & Lemma~\ref{claim4}\\
      $\lambda_1$ & $\lambda_0$ & $\lambda_i$
      &Corollary~\ref{cor:wps_weight_divisibility}\eqref{item:wps_weight_divisibility_1}
      & Lemma~\ref{claim8}&&
      $\lambda_i$ & $\lambda_1$ & $\lambda_2$
      &Corollary~\ref{cor:wps_weight_divisibility}\eqref{item:wps_weight_divisibility_1}
      & Lemma~\ref{claim2}\\
      $\lambda_1$ & $\lambda_2$ & $\lambda_i$
      &Corollary~\ref{cor:wps_weight_divisibility}\eqref{item:wps_weight_divisibility_1}
      & Lemma~\ref{claim8}&&
      $\lambda_i$ & $\lambda_1$ & $\lambda_j$
      &Corollary~\ref{cor:wps_weight_divisibility}\eqref{item:wps_weight_divisibility_1}
      & Lemma~\ref{claim4}\\
      $\lambda_1$ & $\lambda_i$ & $\lambda_j$
      &Corollary~\ref{cor:wps_weight_divisibility}\eqref{item:wps_weight_divisibility_1}
      & Lemma~\ref{claim8}&&
      $\lambda_i$ & $\lambda_2$ & $\lambda_j$
      &Corollary~\ref{cor:wps_weight_divisibility}\eqref{item:wps_weight_divisibility_1}
      & Lemma~\ref{claim7}\\
      $\lambda_2$ & $\lambda_0$ & $\lambda_1$ & \multicolumn{1}{c}{--}
      & \multicolumn{1}{c}{--} &&
      $\lambda_i$ & $\lambda_j$ & $\lambda_k$
      &Corollary~\ref{cor:wps_weight_divisibility}\eqref{item:wps_weight_divisibility_1}
      & Lemma~\ref{claim8}\\
      $\lambda_2$ & $\lambda_0$ & $\lambda_i$
      &Corollary~\ref{cor:wps_weight_divisibility}\eqref{item:wps_weight_divisibility_1}
      & Lemma~\ref{claim2}\\
      \bottomrule
    \end{tabular}
    \vspace{0.5em}
    \caption{A summary of the argument in the case where $m>1$.}
    \label{tab:summary_terminal_mgt1}
\end{table}

Suppose now that $m=1$. We argue exactly as for $m>1$, except in
those cases where Lemma~\ref{claim7} is used. We consider these three
cases separately, once again writing $\lambda_i$ for
$\lambda_i^{(m,a)}$. When $\mu_0 = \lambda_0$,~$\mu_1 =
\lambda_1$,~$\mu_2 = \lambda_i$ with $i \geq 3$, we argue as in
Proposition~\ref{canonicalwpsmutationgraph}
case~\ref{canonical_wps_case_3}. When $\mu_0 =\lambda_0$,~$\mu_1
=\lambda_2$,~$\mu_2 =\lambda_i$ with $i \geq 3$, we argue as in
Proposition~\ref{canonicalwpsmutationgraph}
case~\ref{canonical_wps_case_4}. When $\mu_0 =\lambda_i$ with $i \geq
3$,~$\mu_1=\lambda_2$,~$\mu_2 =\lambda_j$ with $j \geq 3$, we have
$\mu_0 =\frac{t_{n-1}}{y_{a_i}}\left(\frac{t_{n-1}}{y_a}+1\right)$,
$\mu_1 = \left(\frac{t_{n-1}}{y_a}+1\right)^2$ and $\mu_2
=\frac{t_{n-1}}{y_{a_j}}\left(\frac{t_{n-1}}{y_a}+1\right)$. Since
$n\geq 5$ we can find $a_k \in \{0,1,\ldots,n-2\} \setminus
\{a,a_i,a_j\}$. Let $p$ be a prime dividing $y_{a_k}$. Then $p$
divides $\mu_0$ and $\mu_2$ but $p$ does not divide $\mu_1$,
contradicting
Corollary~\ref{cor:wps_weight_divisibility}\eqref{item:wps_weight_divisibility_1}.
\end{proof}

  \begin{figure}[ht]
    \centering
\begin{tikzpicture}[grow'=down,sibling distance=5pt]
\Tree 
[.$\Proj{(1, 1, 6, 14, 21)}$ 
	[.{$\bigcirc$} 
		[.$\bigcirc$
			 [.$\bigcirc$
			]
		]
	]
	[.$\bigcirc$ 
		[.$\bigcirc$
			 [.$\bigcirc$
			]
		]
		[.$\bigcirc$
			 [.$\bigcirc$
			]
			 [.$\bigcirc$
			]
		]
	]
	[.{$\bigcirc$}
		[.$\bigcirc$
			 [.$\bigcirc$
			]
			 [.$\bigcirc$
			]
			 [.$\bigcirc$
			]
		]
		[.$\bigcirc$
			 [.$\bigcirc$
			]
		]
		[.$\bigcirc$
			 [.$\bigcirc$
			]
			 [.$\bigcirc$
			]
			 [.$\bigcirc$
			]
			 [.$\bigcirc$
			]
		]
		[.$\bigcirc$
			 [.$\bigcirc$
			]
			 [.$\bigcirc$
			]
		]
	]
]
\end{tikzpicture}    
\caption{The graph of mutations between four-dimensional weighted
  projective spaces to a depth of $3$}
\label{fig:4dwpsmutationgraph}
\end{figure}
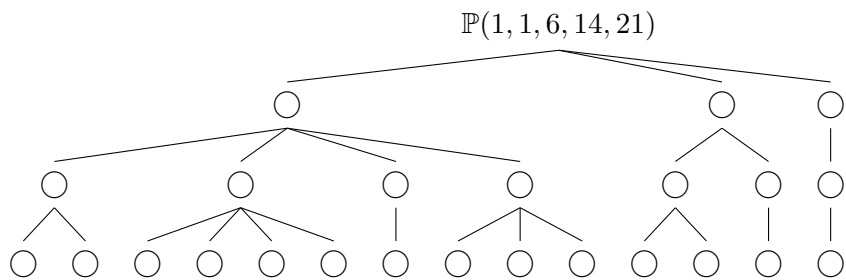
 
\begin{remark}
  The requirement in Proposition~\ref{terminalwpsmutationgraph} that
  $n\geq5$ is necessary. Figure~\ref{fig:4dwpsmutationgraph} shows the
  graph of mutations between four-dimensional weighted projective
  spaces starting from $X_4=\Proj\left(1,1,6,14,21\right)$ to a depth
  of three, where $\bigcirc$ denotes some weighted projective space.
 The rightmost branch of the tree above corresponds to a mutation
  with $\lambda_0=1,\lambda_1=21$ and $\lambda_2=1$. It can be shown
  using the methods of \S\ref{sec:canonical_sings} that this branch
  continues as a chain. We do not know what happens in the other
  branches at greater depth.
\end{remark}

\appendix
\section{The Sylvester Numbers and Related Sequences}

Recall the definition of Sylvester numbers from
Definition~\ref{def:Sylvester}.  Recall also that $t_n := y_n-1$.

\begin{lemma}\label{canonlemma5}\label{sylv++}
  Let $n\geq 1$ and let $i\in\{ 0 ,\ldots, n-1\}$. Then $y_i\divides
  t_n/y_i + 1$ and $y_i\notdivides 2t_n/y_i + 1$.
\end{lemma}
\begin{proof}
  We proceed by induction on $n$. The base case $n=1$ holds
  trivially.  Note that:
  \[
  t_n/y_i + 1 = (t_{n-1}/y_i + 1)y_{n-1} - (y_{n-1}-1)
  \]
  If $i \leq n-2$ then $y_i\divides y_{n-1}-1$ and, by the induction
  hypothesis, $y_i\divides t_{n-1}/y_i + 1$.  Thus $y_i\divides
  t_n/y_i + 1$.  If $i = n-1$ then $t_n/y_i+1 = y_{n-1}$, which is
  certainly divisible by $y_i$.  This completes the induction step,
  proving that $y_i \divides t_n/y_i + 1$.  It follows that $y_i
  \divides 2t_n/y_i + 2$, and so $y_i \notdivides 2t_n/y_i + 1$.
\end{proof}

\begin{lemma}\label{claim1}
  Let $\lambda_i^{(m,a)}$ be as in
  Proposition~\ref{canonicalwpsmutationgraph}.  For
  $i\in\{0,\ldots,n\}$ the sequence $\left(\lambda_i^{(m,a)}\right)_{m\ge 0}$ is
  increasing.
\end{lemma}
\begin{proof}
  This is a straightforward induction on $m$.
\end{proof}
\begin{lemma}\label{claim3}\label{claim5}
  Let $\lambda_i^{(m,a)}$ be as in
  Proposition~\ref{canonicalwpsmutationgraph}. Fix $m \geq 0$.  Then
  $\lambda_0^{(m,a)}$,~$\lambda_1^{(m,a)}$, and~$\lambda_2^{(m,a)}$ are
  pairwise coprime.
\end{lemma}
\begin{proof}
  We begin by showing that $\lambda_1^{(m,a)}$ and $\lambda_2^{(m,a)}$
  are coprime, proceeding by induction. The base case $m=0$ is
  trivial.  Suppose that $\lambda_1^{(m-1,a)}$ and
  $\lambda_2^{(m-1,a)}$ are coprime, and suppose that there exists a
  prime $p$ dividing both $\lambda_1^{(m,a)}$ and
  $\lambda_2^{(m,a)}$. Then:
  \begin{align*}
    p \divides \lambda_1^{(m-1,a)} 
    &&\text{and}&& 
    p\divides \textstyle \frac{(\lambda_1^{(m-1,a)}+\lambda_2^{(m-1,a)})^2}{\lambda_0^{(m-1,a)}}
  \end{align*}
  Thus $p\divides\lambda_2^{(m-1,a)}$, contradicting the hypothesis
  that $\lambda_1^{(m-1,a)}$ and $\lambda_2^{(m-1,a)}$ are coprime. This
  completes the induction step, showing that $\lambda_1^{(m,a)}$ and
  $\lambda_2^{(m,a)}$ are coprime for all $m$.  It follows immediately
  that $\lambda_0^{(m,a)} =\lambda_2^{(m-1,a)}$ and $\lambda_1^{(m,a)}
  = \lambda_1^{(m-1,a)}$ are also coprime for all $m$.

  Suppose now that $p\divides\lambda_0^{(m,a)} = \lambda_2^{(m-1,a)}$.
  Then $p\notdivides\lambda_1^{(m,a)}$, as we have just seen, and so
  $p$ does not divide the numerator of
  \begin{equation*}
    \lambda_2^{(m,a)}
    =\textstyle\frac{\left(\lambda_1^{(m,a)}+\lambda_2^{(m-1,a)}\right)^2}{\lambda_0^{(m-1,a)}}
  \end{equation*}
  Thus $p\notdivides\lambda_2^{(m,a)}$, so $\lambda_0^{(m,a)}$ and
  $\lambda_2^{(m,a)}$ are coprime.
\end{proof}

\begin{lemma}\label{claim2}\label{claim4}\label{claim6}
  Let $\lambda_i^{(m,a)}$ be as in
  Proposition~\ref{canonicalwpsmutationgraph}, let $m\ge 1$, and let
  $a \in \{0,1,\ldots,n-2\}$. There exists a prime $p$ such that $p
  \notdivides \lambda_0^{(m,a)}$, $p \notdivides \lambda_1^{(m,a)}$,
  and $p \divides \lambda_i^{(m,a)}$ for all $i \geq 2$.
\end{lemma}
\begin{proof}
  Lemma~\ref{claim1} implies that, for $m \geq 2$,
  $\lambda_0^{(m-1,a)}=\lambda_2^{(m-2,a)}\le\lambda_2^{(m-1,a)}$.
  Thus
  $\frac{\lambda_1^{(m,a)}+\lambda_2^{(m-1,a)}}{\lambda_0^{(m-1,a)}}>1$
  for all $m \geq 1$, and so there is some prime $p$ such that
  $p\divides \lambda_1^{(m-1,a)}+\lambda_2^{(m-1,a)}$ but
  $p\notdivides\lambda_0^{(m-1,a)}$.  Thus $p\divides\lambda_i^{(m)}$
  for all $i\in\{2,\ldots,n\}$.  Lemma~\ref{claim5} now implies that
  $p \notdivides \lambda_0^{(m,a)}$ and $p \notdivides
  \lambda_1^{(m,a)}$.
\end{proof}

\begin{lemma}\label{claim7}\label{claim9}
  Let $\lambda_i^{(m,a)}$ be as in
  Proposition~\ref{canonicalwpsmutationgraph}, let $m\ge 2$, and let
  $a \in \{0,1,\ldots,n-2\}$. There exists a prime $p$ such that $p
  \notdivides \lambda_1^{(m,a)}$, $p \notdivides \lambda_2^{(m,a)}$,
  and $p \divides \lambda_i^{(m,a)}$ for $i=0$ and all $i \geq 3$.
\end{lemma}
\begin{proof}
  By Lemma~\ref{claim6} there exists a prime $p$ such that $p$ divides
  $\lambda_i^{(m-1,a)}$ for $i\in\{2,\ldots,n\}$ but $p$ does not
  divide $\lambda_0^{(m-1,a)}$ or $\lambda_1^{(m-1,a)}$. Thus
  $p\divides\lambda_i^{(m,a)}$ for all $i\in\{3,\ldots,n\}$ and, as
  $p\divides\lambda_2^{(m-1,a)}=\lambda_0^{(m,a)}$, we have by
  Lemma~\ref{claim5} that $p$ does not divide $\lambda_1^{(m,a)}$ or
  $\lambda_2^{(m,a)}$.
\end{proof}

\begin{lemma}\label{claim11}
  Let $\lambda_i^{(m,a)}$ be as in
  Proposition~\ref{canonicalwpsmutationgraph}, let $m\ge 0$, and let
  $a \in \{0,1,\ldots,n-2\}$. Then $2 \notdivides \lambda_0^{(m,a)}$,
  $2 \notdivides \lambda_2^{(m,a)}$, and $2 \divides
  \lambda_i^{(m,a)}$ for $i=1$ and all $i \geq 3$.
\end{lemma}

\begin{proof}
  This is a straightforward induction on $m$.
\end{proof}

\begin{lemma}\label{claim8}
  Let $\lambda_i^{(m,a)}$ be as in
  Proposition~\ref{canonicalwpsmutationgraph}, let $a \in
  \{0,1,\ldots,n-2\}$, and let $i\in\{3,\ldots,n\}$. Then there exists
  $k>1$ such that for all $m\geq 0$, $k \notdivides\lambda_j^{(m,a)}$
  for $j \in \{0,2,i\}$ and $k \divides \lambda_j^{(m,a)}$ for
  $j\in\{0,\ldots,n\}\setminus\{0,2,i\}$. 
\end{lemma}
\begin{proof}
  We proceed by induction on $m$. For $m=0$ we have
  $\lambda_j^{(m,a)}=2t_{n-1}/y_{a_j}$ for some $a_j$. If $y_{a_i}\neq
  2$ then let $k$ be a prime dividing $y_{a_i}$; otherwise, let
  $k=4$. In either case $k$ has the desired properties and the claim
  holds.

  Suppose now that there exists $k>1$ such that $k
  \notdivides\lambda_j^{(m-1,a)}$ for $j \in \{0,2,i\}$ and $k
  \divides \lambda_j^{(m-1,a)}$ for
  $j\in\{0,\ldots,n\}\setminus\{0,2,i\}$. Lemma~\ref{claim3} implies
  that $k$ and $\lambda_0^{(m-1)}$ are coprime. Thus as
  \begin{align*}
    \textstyle
    \lambda_j^{(m,a)}=\frac{\lambda_j^{(m-1,a)}(\lambda_1^{(m-1,a})+\lambda_2^{(m-1,a)})}{\lambda_0^{(m-1,a)}}
    &&
    \text{for $j\in\{3,\ldots,n\}$}
  \end{align*}
  we see that $k\divides\lambda_j^{(m,a)}$ for
  $j\in\{3,\ldots,n\}\setminus\{i\}$. Since $k \divides
  \lambda_1^{(m,a)} = \lambda_1^{(m-1,a)}$, by Lemma~\ref{claim3}
  again we have that $k \notdivides \lambda_0^{(m,a)}$ and $k
  \notdivides \lambda_2^{(m,a)}$.  Finally as $k$ does not divide
  $\lambda_2^{(m-1,a)}$ but does divide $\lambda_1^{(m-1,a)}$, it does
  not divide $\lambda_1^{(m-1,a)}+\lambda_2^{(m-1,a)}$; since $k$ also
  does not divide $\lambda_i^{(m-1,a)}$, by the recursion formula it
  cannot divide $\lambda_i^{(m,a)}$ either.
\end{proof}

\begin{lemma}\label{claim10}
  Let $\lambda_i^{(m,a)}$ be as in
  Proposition~\ref{canonicalwpsmutationgraph}. There exists a prime
  $p$ such that, for $m\ge 0$, $p\divides\lambda_1^{(m,a)}$ and
  $p\notdivides\lambda_0^{(m,a)}+\lambda_2^{(m,a)}$.
\end{lemma}
\begin{proof} For the case $m=0$, take $p$ to be an odd prime dividing
  $\lambda_1=2t_{n-1}/y_a$. Suppose now that there exists a prime $p$
  such that $p\divides\lambda_1^{(m-1,a)}$ and
  $p\notdivides\lambda_0^{(m-1,a)}+\lambda_2^{(m-1,a)}$.  Then $p
  \notdivides \lambda_1^{(m,a)} = \lambda_1^{(m-1,a)}$ and and
  $p\divides\lambda_1^{(m-1,a)}(\lambda_1^{(m-1,a)}+2\lambda_2^{(m-1,a)})$.
  Now:
  \[
  \textstyle
  \lambda_0^{(m,a)}+\lambda_2^{(m,a)}
  =\frac{\lambda_2^{(m-1,a)}\big(\lambda_0^{(m-1,a)}+\lambda_2^{(m-1,a)}\big)+
    \lambda_1^{(m-1,a)}\big(\lambda_1^{(m-1,a)}+2\lambda_2^{(m-1,a)}\big)}{\lambda_0^{(m-1,a)}}
  \]
  and by Lemma~\ref{claim3}, $p\notdivides\lambda_2^{(m-1,a)}$. Thus
  $p\notdivides\lambda_0^{(m)}+\lambda_2^{(m)}$.  The result follows
  by induction on $m$.
\end{proof}

\begin{lemma}\label{canonlemma1}
  Let $\lambda_i^{(m,a)}$ be as in
  Proposition~\ref{canonicalwpsmutationgraph}, let $b \in
  \{0,1,\ldots,n-2\}$ be such that $b \ne a$, and let $m\geq 0$. Then
  $y_b\divides\lambda^{(m,a)}_2-1$
\end{lemma}

\begin{proof}
  The cases $m=0$ and $m=1$ are straightforward. Suppose now that
  $y_b\divides\lambda^{(m-2,a)}_2-1$ and
  $y_b\divides\lambda^{(m-1,a)}_2-1$.  We have that:
  \[
  \textstyle
  \lambda^{(m,a)}_2-1 = 
  \frac{\lambda_1^{(m-1,a)}\big(\lambda_1^{(m-1,a)}+2\lambda_2^{(m-1,a)}\big)+
    \big(\lambda^{(m-1,a)}_2-1\big)\big(\lambda^{(m-1,a)}_2+\lambda^{(m-1,a)}_0\big)-\big(\lambda^{(m-2,a)}_2-1\big)\lambda^{(m-1,a)}_2}{\lambda_0^{(m-1,a)}}
  \]
  Since $y_{b}\divides 2t_{n-1}/y_a=\lambda_1^{(m,a)}$, we have by
  Lemma~\ref{claim3} that $y_{b}$ and $\lambda_0^{(m-1,a)}$ are
  coprime.  It thus suffices to show that $y_{b}$ divides the
  numerator of the above expression. But this holds by assumption. The
  Lemma follows by induction on $m$.
\end{proof}

\begin{definition}
  Let $\Xn{n}{m}{a}$ be as in \S\ref{sec:canonical_sings} and let
  $h^{(m,a)}$ denote the sum of the weights of $\Xn{n}{m}{a}$.
\end{definition}

\begin{lemma}
  \label{canonlemma6}
  \[
  h^{(m,a)}  =
  \begin{cases}
    2t_{n-1} & \text{if $m=0$} \\
    \frac{\big(\lambda_1^{(m-1,a)}+\lambda_2^{(m-1,a)}\big)h^{(m-1,a)}}{\lambda_0^{(m-1,a)}}
    & \text{otherwise}
  \end{cases} 
  \]
  and, for any $i\in\{3,\ldots,n\}$:
  \[
  h^{(m,a)} =y_{k_i}\lambda_i^{(m,a)}
  \]
\end{lemma}
\begin{proof}
  This is a straightforward calculation.
\end{proof}

\begin{lemma}\label{canonlemma2}\label{canonlemma8}
  Let $\lambda_i^{(m,a)}$ be as in
  Proposition~\ref{canonicalwpsmutationgraph}, let $a \in
  \{0,1,\ldots,n-2\}$, and let $m\geq 1$. Then:
  \[
  \frac{\lambda_0^{(m,a)}\lambda_1^{(m,a)}}{h^{(m,a)}}<\frac{1}{2t_{n-1}}
  \]
\end{lemma}
\begin{proof}
  By Lemma~\ref{canonlemma6} it suffices to prove that
  $\frac{2t_{n-1}\lambda_1^{(m,a)}}{y_{k_i}}<\frac{\lambda_i^{(m,a)}}{\lambda_0^{(m,a)}}$.
  This evidently holds for $m=1$, and
  \[
  \frac{\lambda_i^{(m,a)}}{\lambda_0^{(m,a)}}
  =\frac{\lambda_i^{(m-1,a)}}{\lambda_0^{(m-1,a)}}\frac{(\lambda_1^{(m-1,a)}+\lambda_2^{(m-1,a)})}{\lambda_2^{(m-1,a)}}\\
  >\frac{\lambda_i^{(m-1,a)}}{\lambda_0^{(m-1,a)}}
  \]
  The result follows by induction on $m$.
\end{proof}

\begin{lemma}\label{canonlemma7}
  Let $\lambda_i^{(m,a)}$ be as in
  Proposition~\ref{canonicalwpsmutationgraph}, let $a \in
  \{0,1,\ldots,n-2\}$, let $m\geq 0$, and let $i\in\{ 3,\ldots,
  n\}$. Then:
  \begin{equation*} 
    y_{k_i}\divides\lambda_2^{(m,a)} -\lambda_0^{(m,a)}
  \end{equation*}
\end{lemma}
\begin{proof}
  We proceed by induction on $m$. The base case $m=0$ is trivial.
  Suppose now that $y_{k_i}\divides\lambda_2^{(m-1,a)}
  -\lambda_0^{(m-1,a)}$.  Then, since $\lambda_2^{(m,a)}
  -\lambda_0^{(m,a)} = (\lambda_2^{(m,a)} -\lambda_0^{(m-1,a)} ) -
  (\lambda_2^{(m-1,a)} -\lambda_0^{(m-1,a)} )$, it suffices to show
  that $y_{k_i}$ divides $\lambda_2^{(m,a)}
  -\lambda_0^{(m-1,a)}$. But:
  \[
  \lambda_2^{(m,a)} -\lambda_0^{(m-1,a)}
  =
  \frac{\lambda_1^{(m-1,a)} (\lambda_1^{(m-1,a)} +
    2\lambda_2^{(m-1,a)}) + (\lambda_2^{(m-1,a)}
    -\lambda_0^{(m-1,a)})(\lambda_2^{(m-1,a)} +\lambda_0^{(m-1,a)})}
  {\lambda_0^{(m-1,a)}}
  \]
  Now $y_{k_i}$ divides $\lambda_1^{(m,a)} = \frac{2t_{n-1}}{y_a}$ and
  $y_{k_i}\divides (\lambda_2^{(m-1,a)} -\lambda_0^{(m-1,a)})$, so
  $y_{k_i}$ divides the numerator here. Lemma~\ref{claim3} implies
  that $y_{k_i}$ is coprime to the denominator. Thus $y_{k_i}$ divides
  $\lambda_2^{(m,a)} -\lambda_0^{(m-1,a)}$.
\end{proof}

\begin{lemma}
  \label{lem:not_canonical}
  Let $\lambda_i^{(m,a)}$ be as in
  Proposition~\ref{canonicalwpsmutationgraph}, let $a \in
  \{0,1,\ldots,n-2\}$, let $m \geq 1$, and let
  $\kappa^{(m,a)}=h^{(m,a)}-(\lambda_1^{(m-1,a)}+\lambda_2^{(m-1,a)})$.
  Then $\kappa^{(m,a)}\in\{2,\ldots,h^{(m,a)}-2\}$, and:
  \[
  \sum_{i=0}^{n}\left\{\frac{\lambda_i^{(m,a)}\kappa^{(m,a)}}{h^{(m,a)}}\right\}
  > n-1
  \]
  where $\{x\}$ denotes the fractional part of $x$.
\end{lemma}

\begin{proof}
  The first statement follows immediately from Lemma~\ref{claim1}.  We
  claim that:
  \begin{equation}
    \label{not_canonical_claim}
    \left\{\dfrac{\kappa^{(m,a)}\lambda_0^{(m,a)}}{h^{(m,a)}}\right\}
    =\begin{cases} 1 - \frac{1}{2t_{n-1}} &\text{if $m$ is odd}\\
      1 - \frac{1}{2t_{n-1}} - \frac{1}{y_a} &\text{if $m$ is even}
    \end{cases}
  \end{equation}
  Note that:
  \begin{align} 
    \dfrac{\kappa^{(m,a)}\lambda_0^{(m,a)}}{h^{(m,a)}} & 
    = 1 -\dfrac{(\lambda_1^{(m-1,a)}+\lambda_2^{(m-1,a)})\lambda_0^{(m,a)}}{y_{k_i}\lambda_i^{(m,a)}}
    \qquad\text{by Lemma~\ref{canonlemma6}} \notag \\
    & = 1
    -\dfrac{\lambda_0^{(m-1,a)}\lambda_0^{(m,a)}}{y_{k_i}\lambda^{(m-1,a)}_i}
    \notag \\
    & = 1
    -\dfrac{\lambda^{(m-2,a)}_0\lambda^{(m-1,a)}_0\lambda^{(m-1,a)}_2}{y_{k_i}\lambda^{(m-2,a)}_i
      (\lambda^{(m-2,a)}_1 +\lambda^{(m-2,a)}_2)} \notag\\
    & = 1 -\dfrac{\lambda_2^{(m-2,a)} (\lambda_1^{(m-2,a)}
      +\lambda^{(m-2,a)}_2)}{y_{k_i}\lambda^{(m-2,a)}_i} \notag\\
    & = 1 -\dfrac{\lambda_0^{(m-2,a)} (\lambda_1^{(m-3,a)}
      +\lambda^{(m-3,a)}_2)}{y_{k_i}\lambda^{(m-2,a)}_i} \notag \\
    & \qquad 
    -\dfrac{\lambda^{(m-2,a)}_2 (\lambda_1^{(m-2,a)} +\lambda^{(m-2,a)}_2)
      -\lambda^{(m-2,a)}_0 (\lambda_1^{(m-2,a)}
      +\lambda^{(m-2,a)}_0)}{y_{k_i}\lambda^{(m-2,a)}_i}
    \label{this_term}
  \end{align}
  We claim that the last term \eqref{this_term} here is an
  integer. It is equal to:
  \[
  \frac{(\lambda^{(m-2,a)}_2 -\lambda^{(m-2,a)}_0)(\lambda^{(m-2,a)}_0
    +\lambda_1^{(m-2,a)} +\lambda^{(m-2,a)}_2)}{h^{(m-2,a)}}
  \]
  Now:
  \begin{align*}\dfrac{\lambda^{(m-2,a)}_0 +\lambda_1^{(m-2,a)}
      +\lambda^{(m-2,a)}_2}{h^{(m-2,a)}} 
    & =\dfrac{ h^{(m-2,a)} -\sum_{i=3}^n\lambda^{(m-2,a)}_i}{h^{(m-2,a)}}\\
    & = 1 -\sum_{i=3}^n\dfrac{\lambda_i^{(m-2,a)}}{h^{(m-2,a)}}\\
    & = 1 -\sum_{i=3}^n \frac{1}{y_{k_i}} && \text{by Lemma~\ref{canonlemma6}}\\
    & = 1 -\sum_{i=0}^{n-2} \frac{1}{y_i} + \frac{1}{y_a}\\
    & = \frac{1}{t_{n-1}} + \frac{1}{y_a} && \text{ by Lemma~\ref{lem:sylv}}
  \end{align*}
  Hence \eqref{this_term} is equal to:
  \begin{equation*}
    \frac{(\lambda^{(m-2,a)}_2 -\lambda^{(m-2,a)}_0)
      (t_{n-1}/y_a + 1)}{t_{n-1}}
  \end{equation*}
  Recall that $t_{n-1} =\prod_{i=0}^{n-2} y_i$. We will show that each
  $y_i$ divides the numerator of this expression. By
  Lemma~\ref{canonlemma5} we have that $y_a\divides
  \frac{t_{n-1}}{y_a} + 1$, and by Lemma~\ref{canonlemma7} we have
  that, for all $i \ne a$, $y_i\divides\lambda_2^{(m-2,a)}
  -\lambda_0^{(m-2,a)}$. Hence \eqref{this_term} is an integer. Thus:
  \begin{equation*}
    \left\{\dfrac{\kappa^{(m,a)}\lambda_0^{(m,a)}}{h^{(m,a)}}\right\} =\left\{ 1 -\dfrac{\lambda_0^{(m-2,a)} (\lambda_1^{(m-3,a)} +\lambda^{(m-3,a)}_2)}{y_{k_i}\lambda^{(m-2,a)}_i}\right\} =\left\{\dfrac{\kappa^{(m-2,a)}\lambda_0^{(m-2,a)}}{h^{(m-2,a)}}\right\}
  \end{equation*}
  Since \eqref{not_canonical_claim} holds for $m=1$ and $m=2$, by
  induction it holds for all $m$.

  Lemma~\ref{canonlemma6} implies that
  \begin{equation*}
    \dfrac{\kappa^{(m,a)}}{h^{(m,a)}}=1-\dfrac{\lambda_0^{(m-1,a)}}{h^{(m-1,a)}}
  \end{equation*}
  We have:
  \begin{align*}
    \left\{\dfrac{\kappa^{(m,a)}\lambda_1^{(m,a)}}{h^{(m,a)}}\right\}
    & =\left\{1-\dfrac{\lambda_0^{(m-1,a)}\lambda_1^{(m,a)}}{h^{(m-1,a)}}\right\}
    \\
    & = 1-\dfrac{\lambda_0^{(m-1,a)}\lambda_1^{(m,a)}}{h^{(m-1,a)}}
    && \text{by Lemma~\ref{canonlemma2}}
    \intertext{and for $i\in\{3,\ldots\,n\}$ we have:} 
    \left\{\dfrac{\kappa^{(m,a)}\lambda_i^{(m,a)}}{h^{(m,a)}}\right\}&=
    \left\{1-\dfrac{\lambda_i^{(m-1,a)}(\lambda_1^{(m-1,a)}+\lambda_2^{(m-1,a)})}{h^{(m-1,a)}}\right\}\\
    &=\left\{1-\dfrac{\lambda_1^{(m-1,a)}+\lambda_2^{(m-1,a)}}{y_{k_i}}\right\}
    && \text{by Lemma~\ref{canonlemma6}}\\
    &=\left\{1-\dfrac{\lambda_1^{(m-1,a)}}{y_{k_i}}-\dfrac{\lambda_2^{(m-1,a)}-1}{y_{k_i}}-\dfrac{1}{y_{k_i}}\right\}\\
    & =1-\dfrac{1}{y_{k_i}}
    && \text{by Lemma~\ref{canonlemma1}}
    \intertext{Putting this all together, for $m=1$ we obtain:}
    \sum_{i=0}^{n}\left\{\frac{\kappa^{(1,a)}\lambda_i^{(1,a)}}{h^{(1,a)}}\right\}
    &=\left\{\frac{\kappa^{(1,a)}\lambda_2^{(1,a)}}{h^{(1,a)}}\right\}+n -\sum_{i=0}^{n-2}\dfrac{1}{y_i}-\dfrac{1}{2t_{n-1}}\\
    &=\left\{\frac{\kappa^{(1)}\lambda_2^{(1)}}{h^{(1)}}\right\}+n
    -\dfrac{t_{n-1}-1}{t_{n-1}}-\dfrac{1}{2t_{n-1}}
    && \text{by Lemma~\ref{lem:sylv}}\\
    &>n-1
  \end{align*}
  and for $m\geq 2$ we obtain:
  \begin{align*}
    \sum_{i=0}^{n}\left\{\frac{\kappa^{(m,a)}\lambda_i^{(m,a)}}{h^{(m,a)}}\right\}&
    \geq\left\{\frac{\kappa^{(m,a)}\lambda_2^{(m,a)}}{h^{(m,a)}}\right\}+n -\sum_{i=0}^{n-2}\dfrac{1}{y_i}-\dfrac{1}{2t_{n-1}}-\dfrac{\lambda_0^{(m-1,a)}\lambda_1^{(m,a)}}{h^{(m-1,a)}}\\
    &\geq (n -1)+\dfrac{1}{2t_{n-1}}-\dfrac{\lambda_0^{(m-1,a)}\lambda_1^{(m,a)}}{h^{(m-1,a)}}\\
    &>n-1
  \end{align*}
  where at the last step we used Lemma~\ref{canonlemma8}. 
\end{proof}

\begin{lemma}\label{claim11term}
  Let $\lambda_i^{(m,a)}$ be as in
  Proposition~\ref{terminalwpsmutationgraph}, let $n\ge 5$, let $a \in
  \{0,1,\ldots,n-2\}$, and let $i$,~$j\in\{3,\ldots,n\}$ be
  distinct. There exists a prime $p$ such that, for all $m\geq 0$, $p$
  divides $\lambda_1^{(m,a)}$, $\lambda_i^{(m,a)}$, and
  $\lambda_j^{(m,a)}$ but $p$ does not divide $\lambda_0^{(m,a)}$ or
  $\lambda_2^{(m,a)}$.
\end{lemma}

\begin{proof}
  We have
  $\lambda_i^{(0,a)}=\frac{t_{n-1}}{y_{k_i}}$,~$\lambda_j^{(0,a)}=\frac{t_{n-1}}{y_{k_j}}$,
  and $\lambda_1^{(0,a)}=\frac{t_{n-1}}{y_a}$. Since $n\geq 5$ we can
  find $k_l \in \{0,1,\ldots,n-2\} \setminus \{a,k_i,k_j\}$.  Let $p$
  be a prime dividing $y_{k_l}$. Then $p$ divides
  $\lambda_i^{(0,a)}$,~$\lambda_j^{(0,a)}$, and $\lambda_1^{(0,a)}$
  and thus, by Lemma~\ref{claim3}, $p$ divides neither
  $\lambda_0^{(0,a)}$ nor $\lambda_2^{(0,a)}$.  We now proceed by
  induction on $m$: the induction step is straightforward.
\end{proof}

\begin{lemma}
  \label{lem:not_terminal}
  Let $\lambda_i^{(m,a)}$ be as in
  Proposition~\ref{terminalwpsmutationgraph}, let $a \in
  \{0,1,\ldots,n-2\}$, let $m \geq 1$, and let
  $\kappa^{(m,a)}=h^{(m,a)}-h^{(m-1,a)}$.  Then
  $\kappa^{(m,a)}\in\{2,\ldots,h^{(m,a)}-2\}$, and:
  \[
  \sum_{i=0}^{n}\left\{\frac{\lambda_i^{(m,a)}\kappa^{(m,a)}}{h^{(m,a)}}\right\}
  < 2
  \]
  where $\{x\}$ denotes the fractional part of $x$.
\end{lemma}

\begin{proof}
  The first statement follows immediately from Lemma~\ref{claim1}.
  The conclusions of Lemma~\ref{canonlemma6} hold here too, and thus:
  \begin{equation*}
     \frac{\kappa^{(m,a)}}{h^{(m,a)}}=1-\frac{\lambda_0^{(m-1,a)}}{\lambda_1^{(m-1,a)}+\lambda_2^{(m-1,a)}}
  \end{equation*}
  It follows that:
  \begin{align*}
    \left\{\frac{\kappa^{(m,a)}\lambda_2^{(m,a)}}{h^{(m,a)}}\right\}&=
    \left\{-\lambda_2^{(m,a)}\frac{\lambda_0^{(m-1,a)}}{\lambda_1^{(m-1,a)}+\lambda_2^{(m-1,a)}}\right\}
    = 0 \intertext{and, for $i\in\{3,\ldots,n\}$:}
    \left\{\frac{\kappa^{(m,a)}\lambda_i^{(m,a)}}{h^{(m,a)}}\right\}&
    =\left\{-\lambda_i^{(m,a)}\frac{\lambda_0^{(m-1,a)}}{\lambda_1^{(m-1,a)}+\lambda_2^{(m-1,a)}}\right\}
    = 0
\end{align*}
Thus:
\begin{equation*} 
  \sum_{i=0}^{n}\left\{\frac{\kappa^{(m,a)}\lambda_i^{(m,a)}}{h^{(m,a)}}\right\}
  =\left\{\frac{\kappa^{(m,a)}\lambda_0^{(m,a)}}{h^{(m,a)}}\right\}+\left\{\frac{\kappa^{(m,a)}\lambda_1^{(m,a)}}{h^{(m,a)}}\right\}
  < 2
\end{equation*}
\end{proof}
\newcommand{\etalchar}[1]{$^{#1}$}
\providecommand{\bysame}{\leavevmode\hbox to3em{\hrulefill}\thinspace}
\providecommand{\MR}{\relax\ifhmode\unskip\space\fi MR }
\providecommand{\MRhref}[2]{%
  \href{http://www.ams.org/mathscinet-getitem?mr=#1}{#2}
}
\providecommand{\href}[2]{#2}

\end{document}